\documentclass[notitlepage]{scrartcl}
\usepackage[shortlabels]{enumitem}
\usepackage{amsmath}
\usepackage{amssymb}
\usepackage{amsthm}
\usepackage{abstract}
\usepackage{xcolor}
\usepackage{comment}
\usepackage{mathtools}
\usepackage{graphicx}
\usepackage{caption}
\usepackage{subcaption}
\usepackage{float}
\usepackage{hyperref}
\usepackage{enumitem}
\makeatletter
\def\namedlabel#1#2{\begingroup
   \def\@currentlabel{#2}%
   \label{#1}\endgroup
}
\makeatother
\usepackage{dsfont}
\hypersetup{
    colorlinks=true,
    linkcolor=blue,
    filecolor=magenta,      
    urlcolor=cyan
    }
\makeatletter
\newcommand*{\barfix}[2][.175ex]{%
  \mathpalette{\@barfix{#1}}{#2}%
}
\newcommand*{\@barfix}[3]{%
  \vbox{%
    \kern#1\relax
    \hbox{$#2#3\m@th$}%
  }%
}
\makeatother

\usepackage[capitalise]{cleveref}
\newtheorem{theorem}{Theorem}
\newtheorem{thm}{Theorem}[section]

\newtheorem{lemma}[thm]{Lemma}
\newtheorem{proposition}[thm]{Proposition}

\newtheorem{conjecture}[thm]{Conjecture}
\newtheorem{observation}[thm]{Observation}
\newtheorem{question}[thm]{Question}

\theoremstyle{definition}
\newtheorem{rem}{Remark}
\newtheorem{definition}[thm]{Definition}
\newtheorem{remark}[thm]{Remark}

\newcommand{\footremember}[2]{%
    \footnote{#2}
    \newcounter{#1}
    \setcounter{#1}{\value{footnote}}%
}
\newcommand{\footrecall}[1]{%
    \footnotemark[\value{#1}]%
} 

\newcommand{\whp}{%
    \textbf{whp}
} 
\newcommand{\Whp}{%
    \textbf{Whp}
} 

\usepackage[margin=1in]{geometry} 

\def\cC{\mathcal{C}}

\def\cF{\mathcal{F}}
\def\cP{\mathcal{P}}

\def\cE{\mathcal{E}}

\def\cU{\mathcal{U}}
\def\cS{\mathcal{S}}

\def\e{\mathrm{e}}
\def\eps{\varepsilon}

\def\Aux{\mathrm{Aux}}
\def\ind{\mathds{1}}

\def\agp{\mathrm{AGP}}

\def\interior{\mathrm{Int}}

\title{\vspace{-1.5cm} Nearly spanning cycle in the percolated hypercube} 
\author{%
Michael Anastos \footremember{alley1}{\scriptsize{Institute of Science and Technology Austria (ISTA), Klosterneurburg 3400, Austria. Email:  michael.anastos@ist.ac.at.}}%
\and Sahar Diskin \footremember{alley2}{\scriptsize{School of Mathematical Sciences, Tel Aviv University, Tel Aviv 6997801, Israel. Emails: sahardiskin@mail.tau.ac.il, krivelev@tauex.tau.ac.il.}}%
\and Joshua Erde \footremember{alley3}{\scriptsize{Institute of Discrete Mathematics, Graz University of Technology, Steyrergasse 30, 8010 Graz, Austria. Emails: erde@math.tugraz.at, kang@math.tugraz.at.}}%
\and Mihyun Kang \footrecall{alley3}%
\and Michael Krivelevich \footrecall{alley2}%
\and Lyuben Lichev \footremember{alley4}{\scriptsize{Institute of Statistics and Mathematical Methods in Economics, Technical University of Vienna, A-1040 Vienna, Austria. Email: lyuben.lichev@tuwien.ac.at.}}%
}
\date{\vspace{-1.5cm}}

\begin{document}
\maketitle

\begin{abstract}
Let $Q^d$ be the $d$-dimensional binary hypercube. We form a random subgraph $Q^d_p\subseteq Q^d$ by retaining each edge of $Q^d$ independently with probability $p$. 
We show that, for every constant $\eps>0$, there exists a constant $C=C(\eps)>0$ such that, if $p\ge C/d$, then with high probability $Q^d_p$ contains a cycle of length at least $(1-\eps)2^d$.
This confirms a long-standing folklore conjecture, stated in particular by Condon, Espuny D{\'\i}az, Gir{\~a}o, K{\"u}hn, and Osthus [Hamiltonicity of random subgraphs of the hypercube, Mem. Amer. Math. Soc. 305 (2024), No. 1534].
\end{abstract}

\section{Introduction}
The \textit{$d$-dimensional binary hypercube} $Q^d$ is the graph whose vertex set is $\{0,1\}^d$, and an edge is drawn between two vertices if and only if they differ in a single coordinate.  The \textit{percolated hypercube} $Q^d_p$ is obtained by retaining each edge of $Q^d$ independently and with probability $p$. The hypercube and its random subgraphs arise naturally in many contexts and have received much attention in combinatorics, probability and computer science.

The study of the percolated hypercube $Q^d_p$ was initiated by Sapo\v{z}enko \cite{S67} and by Burtin \cite{B77} who showed that the threshold for connectivity is $p=\frac{1}{2}$. This result was subsequently strengthened by Erd\H{o}s and Spencer \cite{ES79} who further asked whether $Q^d_p$ undergoes a phase transition around $pd=1$ with respect to its component structure similar to the one in the classic binomial random graph $G(n,p)$. Ajtai, Koml\'os, and Szemer\'edi confirmed this~\cite{AKS81}, with subsequent work by Bollob\'as, Kohayakawa, and \L{}uczak \cite{BKL92}, showing that  if $pd<1$ then \textbf{whp}\footnote{Short for {\em with high probability}, meaning with probability tending to one as $d$ tends to infinity.} all the connected components of $Q^d_p$ are small, whereas if $pd>1$ then \textbf{whp} $Q^d_p$ contains a unique giant component containing a linear fraction of the vertices (see \cite{K23} for a simple and self-contained proof). 
Bollob\'as \cite{B90} further showed that, just as in $G(n,p)$, the thresholds in $Q^d_p$ for the disappearance of isolated vertices, for connectivity, and for the existence of a perfect matching all coincide.

In the binomial random graph model, already in their seminal 1960 paper, Erd\H{o}s and R\'enyi~\cite{ER60} asked whether the giant component typically contains a cycle of linear length. This was resolved by Ajtai, Koml\'os, and Szemer\'edi~\cite{AKS81a} who showed that, for any constant $C>1$, when $p=C/n$, $G(n,p)$ typically contains a cycle of length $(1-o_C(1))n$ (see also the independently obtained result of Fernandez de la Vega \cite{Fdlv79}). Later work by Ajtai, Komlós, and Szemerédi~\cite {AKS85}, and independently by Bollobás~\cite {B84}, obtained a hitting time result for the emergence of a Hamilton cycle in the binomial random graph. Given the similarities between the models, a folklore conjecture reiterated over the years (see, for example, the 2014 ICM survey by Frieze~\cite{F14}, \cite[Question 5.5]{EKK22} and~\cite[Question 77]{F19}), is that the percolated hypercube $Q^d_p$ should exhibit similar behaviour with respect to the typical existence of long cycles as the binomial random graph.

The subject was advanced by the recent breakthrough of Condon, Espuny Díaz, Gir\~ao, K\"uhn, and Osthus~\cite{CDGKO21} who showed a hitting time result implying that, in particular, $p=\frac{1}{2}$ is the threshold for the existence of a Hamilton cycle in $Q^d_p$. 
In the same paper, they explicitly conjectured the typical existence of long cycles in the sparser regime.
\begin{conjecture}[Conjecture 1.3 of \cite{CDGKO21}]\label{conj: CDGK021}
Suppose that $p=p(d)$ satisfies that $pd\to\infty$. Then, \textbf{whp}, $Q^d_p$ contains a cycle of length $(1-o(1))\, 2^d$.
\end{conjecture}

In this paper, we confirm this conjecture in a strong quantitative form.
\begin{theorem}\label{thm:main}
For every constant $\eps>0$, there exists a constant $C=C(\eps)>0$ such that, if $p=p(d)\in [0,1]$ satisfies $p\ge C/d$, then \whp $Q^d_p$ contains a cycle of length at least $(1-\eps)\, 2^d$.
\end{theorem}
We remark that the dependence of $C$ on $\eps$ exhibited by our proof of \Cref{thm:main} is inverse polynomial. 

Our proof involves several innovations which, we believe, could be useful in future works. 
Among them is a technique for finding approximate vertex-decompositions of percolated regular bipartite graphs into long paths (Sections~\ref{sec:part 1a} and~\ref{sec: part 1b}), a tree-exploration algorithm in the hypercube (Section \ref{sec:2(a)}), and a coupling with a mixed percolation on the hypercube and monotone paths therein (Section~\ref{sec:2(b)}). See Section~\ref{sec:outline} for an outline of the proof. 

\begin{rem}
\Cref{thm:main} holds in the more general context of mixed percolation. Namely, with minor changes in our arguments, the following statement can be shown: For all constants $\eps,\delta>0$, there exists a constant $C=C(\eps,\delta)>0$ such that, if $p=p(d)\in [0,1]$ satisfies $p\ge C/d$, then \whp the graph $Q^d_p(\delta)$ obtained from $Q^d_p$ after vertex-percolation with parameter $\delta$ contains a cycle of length at least $(1-\eps)\, \delta\, 2^d$ (see Section \ref{sec: notation} for the formal definition of $Q^d_p(\delta)$).
\end{rem}

Let us conclude the introduction with several open problems and conjectures. While we show the typical existence of a nearly-spanning cycle for $p=C/d$ for large enough $C$, it still remains open whether $Q^d_p$ contains a cycle of linear length in the \emph{supercritical} regime when $p=\frac{1+\eps}{d}$ for some small constant $\eps>0$. 
Results in this direction were obtained in \cite{DEKK24,EKK22} where it was shown that, in this regime, $Q^d_{p}$ typically contains a cycle of length $\Omega(2^d/(d\log d))$. 
\begin{conjecture}
Let $\eps>0$ be a constant, and let $p=p(d)=\frac{1+\eps}{d}$. Then, there is a constant $c = c(\eps) > 0$ such that \textbf{whp} $Q^d_p$ contains a cycle of length at least $c\, 2^d$.
\end{conjecture}

Concerning the guaranteed length of the cycle provided in Theorem~\ref{thm:main}, some indications suggest that the bound can be strengthened to $(1-\e^{-\Omega(pd)})\, 2^d$ --- indeed, an analogous bound holds for the length of the longest cycle in the Erd\H{o}s-R\'enyi graph~(\cite{AKS81}, see also \cite{F86} for an accurate estimate). 
Such a lower bound would be optimal since the typical number of isolated vertices is about $\e^{-pd}\, 2^d$.
\begin{conjecture}\label{conj:large}
There are constants $c_0, C_0 > 0$ such that, for $p = p(d)\in [0,1]$ satisfying $pd\ge C_0$, \textbf{whp}  $Q^d_p$ contains a cycle of length at least $(1-\e^{-c_0 pd})\, 2^d$.
\end{conjecture}
Finally, our proof utilises in a key way the `layered' structure of the hypercube. In the context of recent work on percolation on other high-dimensional graphs \cite{CDE24,DEKK24,DEKK22,L22}, it would be interesting to determine if there is a quantitative similarity in the length of the longest cycle after percolation in other high-dimensional graphs without an explicit `layered' structure. As a concrete example, it is not clear if our methods extend to the \emph{middle layer graph} $H[d/2]$, or more generally, other (bipartite) Kneser graphs.
\begin{question}
Fix $p=p(d)=C/d$ for sufficiently large $C>2$. Does $H[d/2]_p$ contain a cycle of length $(1-o_C(1))\, |V(H[d/2])|$ \textbf{whp} ?
\end{question}
We note that, until relatively recently, even the existence of an approximately spanning cycle in $H[d/2]$ was an open problem~\cite{J04}. The existence of a Hamiltonian cycle in $H[d/2]$ was a longstanding conjecture finally settled by M{\"u}tze~\cite{M16,M24}.

The rest of the paper is structured as follows. In Section~\ref{sec: notation}, we lay out notation that will be used throughout the paper. Then, in Section~\ref{sec:outline}, we give an outline of the proof of Theorem~\ref{thm:main}. In Section~\ref{sec:prel}, we prove some auxiliary lemmas that will be useful throughout the paper. 
The proof of Theorem~\ref{thm:main} is spread throughout Sections \ref{sec:part 1a}--\ref{sec: part 3}.

\subsection{Notation}\label{sec: notation}
For real numbers $a,b,c$ with $c > 0$, we write $a=b\pm c$ to mean that $a\in [b-c, b+c]$. For positive integers $k$, we write $[k] = \{1,\ldots,k\}$.
Throughout the paper, all logarithms are taken with respect to the natural base. Any asymptotic notation is used with respect to the parameter $d$. 
We further write $f(C,x) = o_C(x)$ if $\lim_{C\to\infty}f(C,x)/x = 0$, and $f(C,x) = \omega_C(x)$ if $\lim_{C\to\infty}f(C,x)/x =\infty$. With a slight abuse of notation, we sometimes write $a\gg b$ to indicate that $a=\omega(b)$, and $a\ll b$ to indicate that $a=o(b)$. 
Rounding is ignored where irrelevant for better readability.

We introduce some notation specific to high-dimensional hypercubes.  
For simplicity, we write throughout $V = V(Q^d)$.
For a vertex $v\in V$, we write $\ind(v)$ for the set of coordinates of $v$ equal to 1 and call it the \emph{support} of $v$. Given $S\subseteq V$, we let $\ind(S)$ be the \emph{joint support} of $S$, that is, $\ind(S)=\bigcap_{v\in S}\ind(v)$, and $\mathbb{T}(S)$ be that \emph{total support} of $S$, that is, $\mathbb{T}(S) = \bigcup_{v \in S} \ind(v)$.

Moreover, for an integer $i\in [0,d]$, we denote by $L_i\subseteq V$ the set of vertices with exactly $i$ coordinates equal to 1, that is, $L_i\coloneqq \{v\in V\colon |\ind(v)|=i\}$. 
We often call this set \emph{layer $L_i$} or, with a slight abuse of notation, layer $i$ of the (hyper)cube.
For integers $i_1,i_2\in [0,d]$ with $i_1 < i_2$, we also write $L_{i_1,i_2}\coloneqq L_{i_1}\cup L_{i_1+1}\cup \ldots \cup L_{i_2}$.
Further, for subsets $U\subseteq V$, we write $L_i[U]\coloneqq L_i\cap U$ and $L_{i_1,i_2}[U]\coloneqq L_{i_1,i_2}\cap U$. A path in $Q^d$ is called \textit{monotone} if it visits every layer of the hypercube at most once.
Given two vertices $u,v\in V$ such that $\ind(u)\subseteq \ind(v)$, let $Q[u;v]$ be the induced subcube of $Q^d$ whose vertex set is given by $V(Q[u;v])\coloneqq \left\{w\in V\colon \ind(u)\subseteq \ind(w)\subseteq \ind(v) \right\}$. Note that $Q[u;v]$ is isomorphic to a hypercube of dimension $|\ind(v)|-|\ind(u)|$. 

Given $p,q\in [0,1]$, we denote by $Q^d_p$ the random subgraph obtained by retaining every edge of $Q^d$ independently and with probability $p$, and we further write $Q^d_p(q)$ for the random (vertex) subgraph of $Q^d_p$ obtained by retaining every vertex $v\in V$ independently and with probability $q$.
Finally, we denote by $Q_0$ (resp. $Q_1$) the subcube of $Q^d$ of dimension $d-1$ obtained by fixing the first coordinate to be $0$ (resp. $1$). 

If $\cP$ is a family of vertex-disjoint subgraphs of $Q^d$, then we write $V(\cP) = \bigcup_{P \in \cP} V(P)$.

\subsection{Proof outline}\label{sec:outline}

In this section, we present an overview of the proof of Theorem~\ref{thm:main}. For ease of presentation, we will describe the strategy \emph{deterministically}, although in reality some of the statements will hold instead with (sufficiently) high probability.

We will start by partitioning $V(Q^d)$ into a `main' set $V_1$, containing all but a vanishing in $C$ proportion of the vertices of $Q^d$, and some small random `reservoirs' $V_2,V_3$, which are sufficiently `well-distributed' in $V(Q^d)$, and within each layer $L_i$, to preserve the `typical' degrees of most vertices in each partition class to adjacent layers.

On a high-level, our strategy will be to first cover almost all of $V_1$ with a large family of paths, where by `almost all', both here and in the rest of this section, we mean all but a vanishing in $C$ proportion. We then use the vertices in the reservoirs to `merge' these paths into smaller and smaller families of longer and longer paths, still covering almost all of $V_1$, until they are eventually merged into a single, nearly spanning cycle.

After the initial step, in which we build a family of paths covering almost all the vertices of $V_1$, there are four main merging steps. Since the merging process will rely heavily on the `layered' structure of $Q^d$, it will be useful to keep track of the vertices in or between certain layers. Later on, we will fix layers $m_1 \leq m_2 \leq m_3 \leq d/2 \leq m_4$ such that almost all vertices in $Q^d$ lie in $L_{m_3,m_4}$ and $m_1 = \Theta(\log d) \ll m_2 \ll m_3$.

The main steps are then as follows.

\begin{enumerate}[\arabic*{}.]
    \item We first construct a family $\cP_1$ of paths of length $\omega_C(1)$ covering almost all vertices in $L_{m_3,m_4}[V_1]$ (Section~\ref{sec:part 1a}).
    \item We then use the vertices in $L_{m_3,m_4}[V_2]$ together with a novel variant of the Depth-First-Search algorithm on an auxiliary unbalanced bipartite graph in order to merge most of the total volume of these paths into a family $\cP_2$ of paths of length $\omega_C(d)$ (Section~\ref{sec: part 1b}).
    \item We then introduce a process, which we call Merge-Or-Grow (MOG), in which we iteratively grow expanding trees down from the initial and terminal segments of the paths in $\cP_2$ one layer at a time. When the trees corresponding to different paths meet, we will use them to merge two paths. We first run MOG using the vertices in $L_{m_2,m_4}[V_3]$ and show that any tree which does not merge during this process grows to have \emph{many} leaves in $L_{m_2}[Q_0]$ (Section \ref{sec:2(a)}). 
    \item\label{i:step4} We then run MOG using the vertices of $L_{m_1,m_2}[Q_0]$, and analyse this process using an innovative coupling with a simpler process on a mixed percolated hypercube. This allows us to construct a family $\cP_3$ of at most six paths, each of whose initial and terminal segments are connected by a tree to \emph{many} vertices in $L_{m_1}[Q_0]$ (\Cref{sec:2(b)}).
     \item\label{i:step5} Finally, we use some self-similarity properties of $Q^d$ to find \emph{many} disjoint subcubes of dimension $\Theta(d)$ in $L_{0,m_2}[Q_1]$, whose edges we did not reveal yet, which can be used to merge the paths in $\cP_3$. Since $pd \gg 1$, it is very likely that these merges can be performed inside the giant component of one of these subcubes (\Cref{sec: part 3}).
\end{enumerate}

We note that, for Steps \ref{i:step4} and \ref{i:step5}, it will be important that we have some control over the supports of the leaves in the trees constructed via MOG (see, in particular, Proposition \ref{prop:2(a)}).

\section{Preliminaries}\label{sec:prel}
We collect several lemmas used throughout the paper. The first one is a simplified version of the Kruskal-Katona theorem due to Lov\'asz~\cite{L93}.
\begin{thm}[Lov\'asz' version of the Kruskal-Katona theorem]\label{thm:LKK}
Fix $i\in [d]$, a set $A\subseteq L_i$ and a real number $x$ such that $|A| = \prod_{j=0}^{i-1} \tfrac{x-j}{i-j} =: \tbinom{x}{i}$. Then, $A$ has at least $\tbinom{x}{i-1}$ neighbours in $L_{i-1}$ within $Q^d$.
\end{thm}

The second one is a technical lemma used together with Theorem~\ref{thm:LKK}.

\begin{lemma}\label{l:binomtech}
Consider $x \in \mathbb{R}$, $i \in [d]$ and $\alpha \in \mathbb{N}$ such that $i = \Theta(d)$ and $\binom{x}{i} = \Theta\left(d^\alpha\right)$. 
Then, $x \leq i + \frac{3}{2}\alpha$ and, in particular,
\[
\frac{i}{3\alpha} \binom{x}{i}\le \binom{x}{i-1}.
\]
\end{lemma}
\begin{proof}
Note that, for $x \geq i$, $\binom{x}{i}$ is increasing as a function of $x$, and if $y \geq i + \frac{3}{2}\alpha$, then
\begin{align*}
\binom{y}{i} 
&= \prod_{j=0}^{i-1} \frac{y-j}{i-j} 
\geq \prod_{j=0}^{i-1} \bigg(1 + \frac{3\alpha/2}{i-j}\bigg) \geq \exp\bigg( \left(\frac{3}{2}+o(1)\right)  \alpha \sum_{j=0}^{i-\log i}\frac{1}{i-j}\bigg)\\ 
&\geq \exp\left( \left(\frac{3}{2}+o(1)\right)\alpha \log i \right) = \Omega\left(d^{(3/2+o(1))\alpha}\right),
\end{align*}
using the fact that $1+z \geq \e^{(1+o(1))z}$ for $z = o(1)$ and that $\sum_{j=0}^{i-\log i}\frac{1}{i-j} = (1+o(1)) \log i$. 
In particular, it follows that $x \leq i + \frac{3}{2}\alpha$ and, hence,
\[
\frac{i}{3\alpha}\binom{x}{i} = \frac{x-i+1}{3\alpha}\binom{x}{i-1}\le \binom{x}{i-1}.\qedhere
\]
\end{proof}

The following lemma allows one to find monotone paths in hypercubes after mixed percolation. It generalises \cite[Lemma 4.1]{ADEK24} (which addresses edge-percolation) and is proved similarly. We include a full proof here for the sake of completeness.
\begin{lemma}\label{lem: monotone paths}
Fix $\alpha > \e$, $q\in (\e/\alpha, 1]$, let $D\in \mathbb{N}$ be sufficiently large, and let $\rho = \alpha/D$.
Then, the probability that $Q^D_\rho(q)$ contains a monotone path of length $D$ is at least $D^{-5}$. 
\end{lemma}
\begin{proof}
Let $X$ be the random variable counting the number of monotone paths of length $D$ in $Q^D_p(q)$. 
An application of the Cauchy-Schwarz inequality gives
\begin{equation}\label{e:CS}
    \mathbb{P}(X\ge 1) \mathbb{E}[X^2] =\mathbb{E}[\mathbf{1}_{X\ge 1}] \mathbb{E}[X^2]\ge \mathbb{E}[(\mathbf{1}_{X\ge 1}X)^2] = \mathbb{E}[X]^2.
\end{equation}
Since $\mathbb{E}[X]=D! \rho^D q^{D+1}$, where we account for the $D+1$ vertices and the $D$ edges on a monotone path, it remains to bound $\mathbb{E}[X^2]$ from above. 

Denote by $\Pi$ the set of all monotone paths of length $D$ in $Q^D$. 
Define $\textnormal{id} = \{v_0v_1,\ldots,v_{D-1}v_D\}\in \Pi$, where $v_i\in L_i$ is the vertex whose first $i$ coordinates are 1 (and the last $D-i$ coordinates are 0). 
Then,
\begin{align*}
    \mathbb{E}[X^2]
    &=\sum_{\pi_1,\pi_2\in \Pi}\mathbb{P}(\{\pi_1\subseteq Q^D_\rho(q)\}\wedge \{\pi_2\subseteq Q^D_\rho(q)\})\nonumber\\
    &= \sum_{\pi_1,\pi_2\in \Pi}q^{|V(\pi_1\cup\pi_2)|}\rho^{|E(\pi_1\cup \pi_2)|} = D! \sum_{\pi\in \Pi} q^{|V(\textnormal{id}\cup \pi)|}\rho^{|E(\textnormal{id}\cup \pi)|}.
\end{align*}
For a number $k\in [D]$, we restrict our attention to paths $\pi\in \Pi$ such that $\textnormal{id}\cap \pi$ contains $k$ edges, and let $\ell$ be the number of maximal segments of the path $\textnormal{id}$ which are edge-disjoint from $\pi$. Note that, given these quantities, we have that $|V(\textnormal{id}\cup \pi)|=2D+1-k-\ell$ and $|E(\textnormal{id}\cup \pi)|=2D-k$.

Fix $k,\ell$ as above. Then, there are at most $\binom{D}{2\ell}$ ways to choose where the maximal segments lie. 
Indeed, choosing the first and the last edge of the segments (which must be different) determines its position. 
For every $j\in [\ell]$, by writing $x_j$ for the number of edges in the $j$-th segment, we have that $x_1+\ldots+ x_\ell=D-k$.
Given where these maximal segments lie, there are at most $x_1!\cdot\ldots\cdot x_\ell!$ ways to recover $\pi$. Further, note that, for every $j\in[\ell]$, we have that $x_j\ge2$ and thus $\ell \le (D-k)/2$. Therefore,
\begin{align}
    \mathbb{E}[X^2]&\le D!\sum_{k=0}^{D}\sum_{\ell=0}^{(D-k)/2}\binom{D}{2\ell}q^{2D+1-k-\ell}\rho^{2D-k}\max_{\substack{x_1,\ldots,x_{\ell}\geq 2,\\ x_1+\ldots+x_{\ell}=D-k}} \left\{\prod_{j=1}^\ell x_j!\right\}\nonumber\\
    &\le D!q^{D+1}\rho^D+D!\sum_{k=0}^{D-2}\sum_{\ell=1}^{(D-k)/2}\binom{D}{2\ell}q^{2D-k-\ell}\rho^{2D-k}\max_{\substack{x_1,\ldots,x_{\ell}\geq 2,\\ x_1+\ldots+x_{\ell}=D-k}} \left\{\prod_{j=1}^\ell x_j!\right\},\label{eq: 1-2nd-moment}
\end{align}
where the first summand corresponds to the case where $k=D$ (and $\ell=0$). We now turn to estimate the second summand. By the convexity of the Gamma function, the maximum in \eqref{eq: 1-2nd-moment} is attained when $x_j=2$ for all $j\in [\ell]$ except one, which is equal to $D-k-2(\ell-1)$. Thus,
\begin{align}
    \sum_{k=0}^{D-2}\sum_{\ell=1}^{(D-k)/2}\binom{D}{2\ell}&q^{2D-k-\ell}\rho^{2D-k}\max_{\substack{x_1,\ldots,x_{\ell}\geq 2,\\ x_1+\ldots+x_{\ell}=D-k}} \left\{\prod_{j=1}^\ell x_j!\right\}\nonumber\\
    &\le (q\rho)^{2D}\sum_{k=0}^{D-2}\sum_{\ell=1}^{(D-k)/2}\left(\frac{\e D}{q\ell}\right)^{2\ell}\left(\frac{D}{q\alpha}\right)^{k}(D-k-2(\ell-1))!\,,\label{eq: 2-2nd-moment}
\end{align}
where we used the inequality $\tbinom{D}{2\ell}\le (\e D/2\ell)^{2\ell}$.
By changing the order of summation, and then by changing variables, we have
\begin{align}
   \sum_{k=0}^{D-2} \sum_{\ell=1}^{(D-k)/2}&\left(\frac{\e D}{q\ell}\right)^{2\ell}\left(\frac{D}{q\alpha}\right)^{k}(D-k-2(\ell-1))!\nonumber\le \sum_{\ell=1}^{D/2}\left(\frac{\e D}{q\ell}\right)^{2\ell}\sum_{k=0}^{D-2\ell}\left(\frac{D}{q\alpha}\right)^k(D-k-2\ell+2)!\nonumber
   \\&=\sum_{\ell=1}^{D/2}\left(\frac{\e D}{q\ell}\right)^{2\ell}\left(\frac{D}{q\alpha}\right)^{D-2\ell}\;\sum_{m=0}^{D-2\ell}\left(\frac{q\alpha}{D}\right)^m(m+2)!\,. \label{eq: 3-2nd-moment}
\end{align}
Moreover, by using the inequality $(m+2)!\le D^3(m/\e)^m$ for $m\in [0,D-2]$, the inner sum is bounded from above by
\begin{equation}\label{eq: 4-2nd-moment}
D^3\sum_{m=0}^{D-2\ell}\left(\frac{q\alpha m}{D \e}\right)^m \le D^4 \cdot \left(\frac{q\alpha}{\e}\right)^{D-2\ell}.
\end{equation}
Now, \eqref{eq: 3-2nd-moment} together with \eqref{eq: 4-2nd-moment} give
\begin{align}
\sum_{k=0}^{D-2} \sum_{\ell=1}^{(D-k)/2}&\left(\frac{\e D}{q\ell}\right)^{2\ell}\left(\frac{D}{q\alpha}\right)^{k}(D-k-2(\ell-1))!\nonumber\le\sum_{\ell=1}^{D/2}\left(\frac{\e D}{q\ell}\right)^{2\ell}\left(\frac{D}{q\alpha}\right)^{D-2\ell}D^4 \cdot \left(\frac{q\alpha}{\e}\right)^{D-2\ell}\nonumber\\
&=D^4\cdot\left(\frac{D}{\e}\right)^D\sum_{\ell=1}^{D/2}\left(\frac{\e^2}{q\ell}\right)^{2\ell}. \label{eq: 5-2nd-moment}
\end{align}
Hence, putting together \eqref{eq: 1-2nd-moment}, \eqref{eq: 2-2nd-moment}, and \eqref{eq: 5-2nd-moment}, we have
\begin{align*}
    \mathbb{E}[X^2]\le D!q^{D+1}\rho^D+D!(q\rho)^{2D}D^4\cdot\left(\frac{D}{\e}\right)^D\sum_{\ell=1}^{D/2}\left(\frac{\e^2}{q\ell}\right)^{2\ell}\le D^5(D!q^{D+1}\rho^D)^{2},
\end{align*}
where in the last inequality we used that $\sum_{\ell=1}^{D/2}\left(\frac{\e^2}{q\ell}\right)^{2\ell}$ is dominated by a geometric sum and hence $\left(\frac{D}{\e}\right)^D\sum_{\ell=1}^{D/2}\left(\frac{\e^2}{q\ell}\right)^{2\ell}\le D!\,$. Thus, \eqref{e:CS} implies that $\mathbb{P}\left(X\ge 1\right)\ge D^{-5}$, as desired.
\end{proof}

We will also need two lemmas which assert the existence of many disjoint subcubes covering certain collections of points in $Q^d$. We start with a simple observation.

\begin{observation}\label{o:subcubes}
Suppose that $u_1,u_2,v_1,v_2 \in Q^d$ are such that $\ind(u_1) \subseteq \ind(u_2)$, $\ind(v_1) \subseteq \ind(v_2)$ and $\ind(u_1) \setminus \ind(v_2) \neq \varnothing$. Then, $Q[u_1;u_2]$ and $Q[v_1;v_2]$ are vertex-disjoint.
\end{observation}

We now turn to the two structural lemmas. 
\begin{lemma}\label{lem:firstsubcubes}
Fix $I \subseteq [d]$, $m_1\leq m_2 \leq d$, a set $S \subseteq L_{m_2}$ such that $I \subseteq \ind(S)$ and a set $Z \subseteq L_{m_1}$ such that $\mathbb{T}(Z) \subseteq I$. Then, there exists a set $\{ u(z,s) \colon s \in S, z \in Z\} \subseteq L_{m_1 + m_2 - |I|}$ such that
\begin{enumerate}[\textnormal{(\arabic*)}]
    \item\label{i:top} for each $z \in Z$, the subcubes $\{Q[u(z,s);s] \colon s \in S\}$ are pairwise disjoint;
    \item\label{i:bottom} for each $z \neq z'$ and each pair $s,s' \in S$ (not necessarily distinct), the two subcubes $Q[z;u(z,s)]$ and $Q[z';u(z',s')]$ are disjoint.
\end{enumerate}
\end{lemma}
\begin{proof}
For each $s \in S$ and $z\in Z$, we define $u(z,s)$ to be the vertex with
\begin{equation}\label{e:sz}
\ind(u(z,s)) =  \ind(s) \setminus ( I \setminus \ind(z)).
\end{equation}
In particular, $\ind(z) \subseteq \ind(u(z,s)) \subseteq \ind(s)$. It remains to check that each of the two points holds.

For the former, note that, for every $s \neq s' \in S$, there is some $j\in \ind(s) \setminus \ind(s')$. Hence, since $I \subseteq \ind(S)$, it follows that $j \not\in I$ and so, by \eqref{e:sz}, $j \in \ind(u(z,s))$. Hence, $j \in \ind(u(z,s)) \setminus \ind(s')$ and therefore, by \cref{o:subcubes}, $Q[u(z,s);s]$ and $Q[u(z,s');s']$ are disjoint.

For the latter, note that, for each $z \neq z'$, there is some $j \in \ind(z) \setminus \ind(z')$. Since $\ind(z) \subseteq I$, it follows that $j \in I \setminus \ind(z')$. 
In particular, by \eqref{e:sz}, we see that $j \not\in \ind(u(z',s'))$ and so $j \in \ind(z) \setminus \ind(u(z',s'))$. Hence, by \cref{o:subcubes}, $Q[z;u(z,s)]$ and $Q[z';u(z',s')]$ are disjoint.
\end{proof}

\begin{lemma}\label{lem:secondsubcube}
Fix $K \subseteq [2,d]$ of size $2d/3$ and, for each $k \in K$, let $u_k$ and $v_k$ be the vertices of $Q^d$ such that $\ind(u_k) = \{1,k\}$ and $\ind(v_k) = [d] \setminus (K \setminus \{k\})\supseteq \ind(u_k)$. Then, the subcubes $Q[u_k;v_k]$ are pairwise vertex-disjoint for different $k \in K$; subcubes of $Q_1$ of dimension $d/3-1$; and contained in $L_{1,d/3+1}$.
\end{lemma}
\begin{proof}
The first property follows from \cref{o:subcubes} since for $k \neq k'$ we have $k \in \ind(u_k) \setminus \ind(v_{k'})$. The second and the third properties are straightforward.
\end{proof}

\section{Covering the middle layers with short paths}\label{sec:part 1a}
For the rest of the paper, we fix
\begin{equation}\label{e:layers def}
\begin{aligned}
&m_1 = 50 \log d, \quad m_2 = d/2 - d^{0.7},\\
&m_3 = d/2 - d^{0.6} \quad \text{ and } \quad m_4 = d/2 + d^{0.6}.    
\end{aligned}
\end{equation}
Note that $m_2,m_3,m_4 = (1+ o(1)) d/2$ and that $|L_{m_3,m_4}| = (1+o(1)) |V(Q^d)|$. Furthermore, we will assume for convenience that each $m_i$ is even.

Our construction will proceed `layer-by-layer' through the cube, gradually exposing both the edges of $Q^d_p$ and a randomly chosen partition $(V_1,V_2,V_3)$ of the vertex set. 
The said partition will be constructed in a way ensuring that it is sufficiently well-distributed through each layer.

\begin{definition}\label{def:spread partition}
Fix $q_2 = q_3 = C^{-1/80}$ and $q_1 = 1 - 2C^{-1/80}$. 
Given an even $i \in [m_3,m_4]$, we say that a partition $(V_1,V_2,V_3)$ of $L_{i,i+1}$ is \emph{well-spread} if each of the following holds.
\begin{itemize}
    \item $|L_j[V_k]| = (1 \pm C^{-1}) q_k |L_j|$ for each $j\in \{i,i+1\}$ and $k\in \{1,2,3\}$.
   \item For each $\{j,j'\} = \{i,i+1\}$ and $k\in \{1,2,3\}$,
   \[
   \left| \left\{ v \in L_j \colon d(v,L_{j'}[V_k]) \neq (1 \pm C^{-1})q_k d/2\right\} \right| = o(d^{-1}|L_j|).
   \]
\end{itemize}
Furthermore, we denote by $V_{bad} \subseteq L_{i,i+1}$ the vertices whose degree towards some of the sets $V_1,V_2,V_3$ on a neighbouring layer is `far' from its mean, that is, vertices that do not satisfy the second point above. Note that in a well-spread partition, $|V_{bad}| = o(|L_{i,i+1}|)$.
\end{definition}

\begin{remark}\label{r:wellspread}
For any even $i \in [m_3,m_4]$, standard applications of the Chernoff bound show that, if we choose the partition $(V_1,V_2,V_3)$ by assigning each vertex to the partition class $V_j$ with probability $q_j$ independently, then the resulting partition of $L_{i,i+1}$ is well-spread with probability $1- o(1/d)$.
\end{remark}

In this section, we will show that, for each even $i \in [m_3,m_4]$, given a well-spread partition $(V_1,V_2,V_3)$ of $L_{i,i+1}$, \whp we can construct a family $\cP_1(i)$ of disjoint paths of length $\omega_C(1)$ in $Q^d_p[L_{i,i+1}[V_1]]$ which cover a $(1-o_C(1))$ proportion of the vertices of $L_{i,i+1}[V_1]$. 
In particular, \whp the union of $\cP_1(i)$ over all even $i \in [m_3,m_4]$ will cover a $(1-o_C(1))$ proportion of $V(Q^d)$.

For this construction, and for later steps, it will be important that these paths do not contain any of the vertices with irregular degree to some partition class, that is, those in $V_{bad}$.

\begin{lemma}\label{lem:paths_in_M}
Fix even $i \in [m_3,m_4]$ and a well-spread partition $(V_1,V_2,V_3)$ of $L_{i,i+1}$.
Then, with probability $1-o(1/d)$, $Q^d_p[L_{i,i+1}[V_1\setminus V_{bad}]]$ contains a family $\cP_1(i)$ of vertex-disjoint paths of length between $C^{1/6}$ and $2C^{1/6}$ which cover a total of at least $(1-C^{-1/5})|L_{i,i+1}[V_1]|$ vertices.
\end{lemma}

The proof of \cref{lem:paths_in_M} is relatively simple. First, since the partition is well-spread, the graph $H=Q^d[L_{i,i+1}[V_1]]$ is close to being regular. 
Hence, when $C$ is large, \whp this will also be true after percolation with probability $p$. 
It will then follow from K\"onig's theorem (see, for example, \cite[Theorem 7.1.7]{W96}) that \whp $H_p$ contains two disjoint matchings $M_1, M_2$ covering almost all of its vertices, and the union $M_1 \cup M_2$ covers $(1-o_C(1))$ proportion of the vertices with long cycles and paths.

We delay the proof of \cref{lem:paths_in_M} and first show a couple of preparatory lemmas.

\begin{lemma}\label{lem:matchings}
Fix even $i \in [m_3,m_4]$ and a well-spread partition $(V_1,V_2,V_3)$ of $L_{i,i+1}$. Then,~with probability $1-o(1/d)$, $Q^d_p[L_{i,i+1}[V_1\setminus V_{bad}]]$ contains two edge-disjoint matchings $M_1$ and $M_2$, each covering at least $(1-4C^{-2/5})|L_{i,i+1}[V_1]|$ vertices. 
\end{lemma}
\begin{proof}
Let $H\coloneqq Q^d[L_{i,i+1}[V_1]]$, $\delta\coloneqq C^{-2/5}$, $\Delta = (1+C^{-1})q_1d/2$ and $h\coloneqq |V(H)|$. Since the partition is well-spread, by \cref{def:spread partition}, every vertex in $V(H) \setminus V_{bad}$ has degree (in $H$) at most $\Delta$. Furthermore, 
\begin{equation}\label{e:edgesH}
(1-2C^{-1})q_1 dh/4 \leq|V(H) \setminus V_{bad}| (1-C^{-1})q_1 d/4 \leq |E(H)| \leq dh.
\end{equation}

Let $U$ be the set of vertices in $V(H) \setminus V_{bad}$ whose degree in $H_p$ is larger than
$(1+\delta) \Delta p$, and $E_U$ be the set of edges in $H_p$ incident to $U$.
For any edge $e = uv\in E(H)$, conditionally on $e\in E(H_p)$, $e$ belongs to $E_U$ only if either $u$ or $v$ is in $V(H) \setminus V_{bad}$ and is incident to at least $(1+\delta) \Delta p-1$ additional edges in $H_p$.
Hence, by the Chernoff bound,
\begin{align*}
\mathbb P(e\in E_U) 
&\le 2p\mathbb P(\mathrm{Bin}(\Delta,p)\ge (1+\delta) \Delta p-1)\\
&\le 2p\mathbb P(\mathrm{Bin}(\Delta,p) - \Delta p\ge C^{3/5}/4)\le 2p\exp(-C^{1/5}/20).
\end{align*}

In particular,
\begin{equation}\label{eq:exps}
\mathbb E[|E_U|] \le 2\exp(-C^{-1/5}/20) \mathbb E[|E(H_p)|].
\end{equation}

Furthermore, Chernoff's bound together with \eqref{e:edgesH} imply that, with probability $1-o(1/d)$, we have 
\[
|E(H_p)| \geq (1+o(1))\mathbb E[|E(H_p)|] \geq (1-3C^{-1})q_1 pdh/4.
\]
Moreover, note that adding/removing an edge to/from $H_p$ changes $|U|$ by at most two and $|E_U|$ by at most $2(1+\delta) \Delta p$. 
Thus, the Azuma-Hoeffding inequality (see, e.g., \cite[Chapter~7]{AS16}) together with~\eqref{e:edgesH} and~\eqref{eq:exps} imply that, with probability $1-o(1/d)$, we have $|E_U| \le C^{-1}\mathbb E[|E(H_p)|]$.

Now, K\"onig's theorem \cite[Theorem 7.1.7]{W96} implies the existence of a proper edge-colouring of $H_p\setminus U$ partitioning the edge set $E(H_p\setminus U)$ into $(1+\delta) \Delta p$ colour classes, each being a matching. The two largest matchings, $M_1$ and $M_2$, satisfy 
\begin{align*}
    |M_1|+|M_2|\ge \frac{2}{(1+\delta)\Delta p}(|E(H_p)|-|E_U|) \ge \frac{2(1-2C^{-1})|E(H_p)|}{(1+\delta)\Delta p}\ge (1-2\delta)h.
\end{align*}
Since $|M_1|,|M_2|\le h/2$, we that $|M_1|,|M_2|\ge (1/2-2\delta)h$, and thus each of $M_1,M_2$ cover at least $(1-4\delta)h$ vertices, as desired. 
\end{proof}
Since there are relatively few short cycles in $Q^d$, a simple first moment argument will show that \whp almost none of the vertices of $H_p$ are contained in cycles of length at most $2C^{1/6}$. Thus, typically almost all of the vertices covered by $M_1 \cup M_2$ are contained in cycles or paths of length between $C^{1/6}$ and $2C^{1/6}$.
\begin{lemma}\label{lem:short cycles}
Fix $i \in [m_3,m_4]$. Then,  with  probability $1-o(1/d)$, there are $o(|L_{i,i+1}|)$ cycles of length at most $2C^{1/6}$ in $Q^d_p[L_{i,i+1}]$.
\end{lemma}
\begin{proof}
Fix $\ell\coloneqq C^{1/6}$. We note that every cycle in $Q^d$ must contain an even number of edges along each coordinate and, in particular, must have even length. 
Thus, given a vertex $v \in L_{i,i+1}$, there are at most $\binom{d}{k}(2k)!\le d^k(2k)!$ cycles of length $2k$ containing $v$ (indeed, every coordinate must participate twice in a cycle, and there are at most $(2k)!$ ways to order the changes in the coordinates). Hence, the expected number of cycles of length $2k\le 2\ell$ is dominated by
\[
    \sum_{k=2}^{\ell} |L_{i,i+1}|\cdot 2k\cdot d^k(2k)!\cdot \left(\frac{C}{d}\right)^{2k}=O\left(\frac{|L_{i,i+1}|}{d^2}\right),
\]
and the conclusion follows from Markov's inequality.
\end{proof}

We now prove \cref{lem:paths_in_M}.
\begin{proof}[Proof of \cref{lem:paths_in_M}]
As before, fix $\ell\coloneqq C^{1/6}$, $H\coloneqq Q^d[L_{i,i+1}[V_1]]$ and $h \coloneqq |V(H)|$. Note that, since the partition is well-spread, $h \geq (1-C^{-1})q_1 |L_{i,i+1}|$.

By \cref{lem:matchings,lem:short cycles}, with probability $1-o(1/d)$, there exist two edge-disjoint matchings $M_1$ and $M_2$, each covering at least $(1-4C^{-2/5})|L_{i,i+1}[V_1]|$ vertices and there are at most $o(|L_{i,i+1}|) = o(h)$ cycles of length at most $2C^{1/6}$ in $Q^d_p[L_{i,i+1}]$.

Note that $M_1 \cup M_2$ is a disjoint union of paths and cycles, and let $A$ (resp. $B$) be the set of vertices in cycles (resp. paths) in $M_1\cup M_2$ of length at most $2\ell$. Since $|A| = o(h)$ by assumption, it remains to bound $|B|$ from above. 

Note that a vertex in $B$ must be contained in a path whose endpoints belong to only one of the matchings $M_1$ and $M_2$. Since there are at most $8C^{-2/5}h$ such vertices, $|B|\le (2\ell+1)\cdot 8C^{-2/5}h/2$.
In total, there are at least $\left(1-C^{-1/5}\right)h$ vertices in cycles of length at least $2\ell$ or paths of length at least $2\ell$ in $M_1\cup M_2$. These can be cut into paths of length between $\ell$ and $2\ell$, as required.
\end{proof}

\section{Merging paths via a modified DFS process}\label{sec: part 1b}

In this section, we will show how to use the vertices in $L_i[V_2]$ to merge each family $\cP_1(i)$ into a family of significantly longer paths, of length $\omega_C(d)$, which still cover almost all vertices in each $L_{i,i+1}$ for even $i\in [m_3,m_4]$. More concretely, we will prove the following. 

\begin{proposition}\label{prop: long paths}
Fix even $i \in [m_3,m_4]$ and a well-spread partition $(V_1,V_2,V_3)$ of $L_{i,i+1}$. Then, with probability $1-o(1/d)$, $Q^d_p[L_{i,i+1}[V_1 \cup V_2]]$ contains a family $\cP_2(i)$ of vertex disjoint paths, each of length in the interval $[C^{1/12}d/4, 4C^{1/12}d]$, which together cover at least $(1-C^{-1/800})|L_{i,i+1}|$ vertices. 
\end{proposition}

In the sequel, we condition on the success of \cref{lem:paths_in_M} and the existence of the family of paths $\cP_1(i)$.
The idea will be to identify, for every path $P \in \cP_1(i)$, an initial and a terminal \emph{segment} of length $C^{1/8}$, noting that this is an $o_C(1)$ proportion of the path $P$. We can \emph{merge} two paths $P_1 \neq P_2 \in \cP_1(i)$ by finding a path of length 2 through $L_i[V_2]$ joining a segment of $P_1$ to a segment of $P_2$ (see Figure~\ref{fig: simple merge}). 
\begin{figure}
\centering
\includegraphics[width=0.55\textwidth]{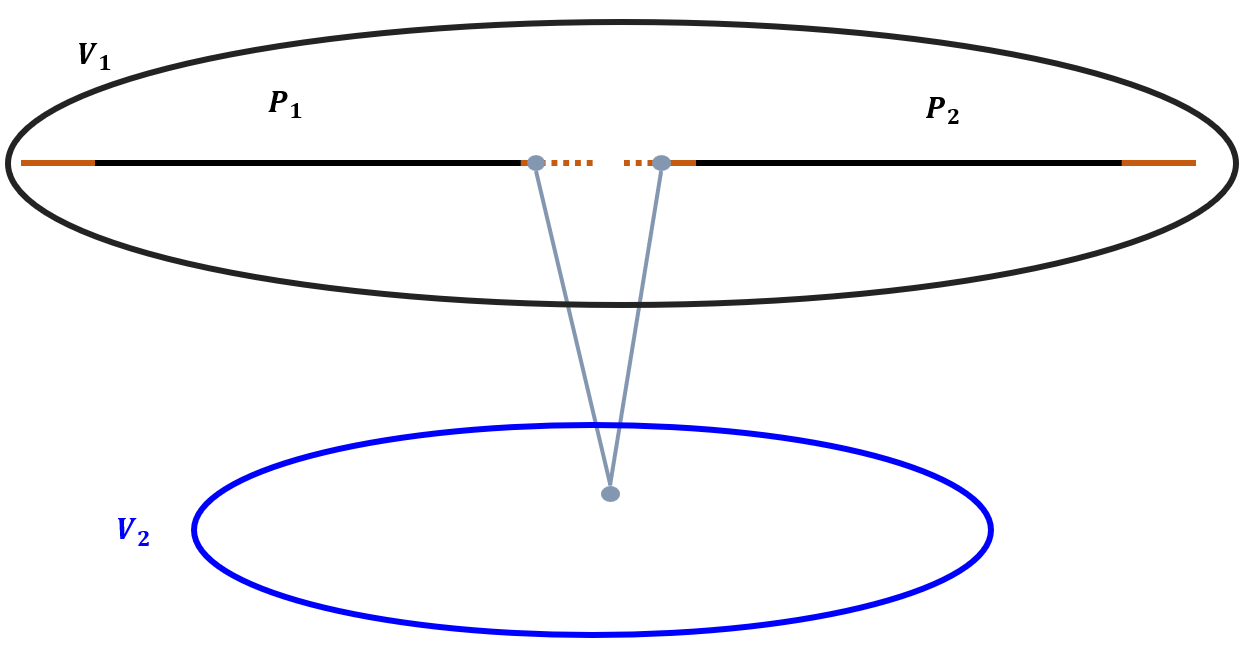}
\caption{Illustration of two paths $P_1, P_2$ merging via a path of length 2 through $V_2$. 
In each of $P_1,P_2$, the initial and the terminal segments $P_i^*$ appear in brown. Note that some vertices in these segments, those on the dashed lines, do not lie in the merged path.}
\label{fig: simple merge}
\end{figure}

We aim to construct paths of length $\omega_C(d)$ by finding sequences of paths $P_1, P_2, \ldots$ which can be consecutively merged in this manner. Since the partition is well-spread, the bipartite graph between $L_{i+1}[V_1]$ and $L_i[V_2]$ has sufficiently good expansion properties allowing us to do this in a relatively `greedy' manner via a Depth-First-Search type process.

More precisely, for all $P \in \cP_1(i)$, we denote by $P^+$ the subpath consisting of the first $2C^{1/8}$ vertices in $P$, and by $P^-$ the subpath consisting of the last $2C^{1/8}$ vertices in $P$. 
We write $\mathcal{S}_1(i) = \{ P^* \colon P \in \cP_1(i), * \in \{+,-\} \}$
for the set of all such subpaths. Note that each $P^* \in \cS_1(i)$ contains $C^{1/8}$ vertices in $L_{i+1}$.

We note that, since the partition is well-spread and each $P \in \cP_1(i)$ is contained in $V_1 \setminus V_{bad}$,
by \cref{def:spread partition}, for every $S \in \mathcal{S}_1(i)$,
\begin{equation}\label{e:segmentnhbr}
|N(S) \cap L_i[V_2]|\ge C^{1/8}\cdot (1-C^{-1})q_2d/2-(C^{1/8})^2\ge C^{1/10}d,
\end{equation}
where we also used the fact that every two vertices have at most one common neighbour in $L_i$.

We construct an auxillary bipartite graph $\Aux$ with parts $A =\mathcal{S}_1(i)$ and $B=L_i[V_2]$ where $S\in A$ is adjacent to $v\in B$ in $\Aux$ if $v$ is adjacent to some vertex in $S$ (in $Q^d$). 

We note some basic properties of the graph $\Aux$. By \cref{def:spread partition,lem:paths_in_M}, we have
\begin{align}
|A| = 2|\cP_1(i)|\ge 2\cdot\frac{(1-C^{-1/5})|L_{i,i+1}[V_1]|}{2C^{1/6}}\ge \frac{|L_{i,i+1}|}{2C^{1/6}} \ge \frac{C^{1/80}|B|}{3C^{1/6}}.\label{eq: A vs B}
\end{align}
Furthermore, by \eqref{e:segmentnhbr}, every $u\in A$ satisfies $d_{\Aux}(u)\ge C^{1/10}d$, and clearly every $v\in B$ satisfies $d_{\Aux}(v) \leq  d$.

To prove \cref{prop: long paths}, we will analyse a modified DFS exploration algorithm on the random subgraph $\Aux_p$ together with a perfect matching $M$ of $A$ consisting of the  pairs $\{P^+,P^-\}$ for all $P\in \cP_1(i)$ (that is, the edges of $M$ are determinstic). Note that every edge in $\Aux$ corresponds to an edge in $Q^d$ between $V_1$ and $V_2$, and in the proof of \cref{lem:paths_in_M} we only exposed the edges in $Q^d_p[V_1]$. 

In the execution of the DFS, every time we visit a vertex $v \in A$, if the DFS has not yet traversed the edge in $M$ which is incident to $v$, then we continue our exploration via this edge. Otherwise, we continue our exploration following the usual DFS algorithm on $\Aux_p$.

We will run this algorithm for several \emph{phases}, each corresponding to an execution of our DFS algorithm on a subgraph of $\Aux_p \cup M$ lasting for $O(d)$ \emph{steps}. We will see (in \cref{l: DFS}) that the expansion properties of $\Aux$ are sufficient to guarantee that each phase likely produces a path of length $\omega_C(d)$ in $Q^d_p[L_{i,i+1}]$.

However, the removal of vertices used in earlier phases might influence the expansion properties of the remaining graph in later phases. An essential part of the algorithm is then a \emph{clean up subroutine} performed between phases, which will delete a small set of vertices to preserve the expansion properties of the graph. 
We will show (in \cref{l: DFS2}) that typically only a $o_C(1)$ proportion of the vertices are deleted during the clean up subroutines, and the paths constructed in the successful rounds cover almost all the vertices of $\Aux$.

We continue with a formal description of the algorithm. 

\subsection{\texorpdfstring{DFS exploration algorithm on $\mathbf{\Aux\cup M}$ (DFS-Aux)}{}} Fix an arbitrary ordering $\sigma$ of $V(\Aux)$ and, for every vertex $v$ in $\Aux$, fix an arbitrary ordering $\sigma_v$ of the edges incident to $v$.
For every edge $e$ in $\Aux$, denote by $X_e$ the indicator random variable of the event that $e$ remains in $\Aux_p$.

Throughout the algorithm DFS-Aux, we maintain five sets: the set of vertices $U_1$ already processed in the current phase; the set $U_2$ which spans the path currently explored; the set $U$ of vertices processed in previous phases; the set $W$ of vertices of too low degree and the set $Z$ of vertices which are yet to be processed. We further maintain a family $\cU$ of paths in $\Aux_p\cup M$. The sets $U_1, U_2, U$ and $W$ are initially empty and gradually grow, while $Z$ contains the vertices yet to be processed; in particular, $Z = V(\Aux)$ in the beginning and decreases throughout.
DFS-Aux will proceed into \emph{phases} further divided into several \emph{steps}.
The sets $U_1,U_2,$ and $Z$ change gradually within every single phase, 
while the sets $U,W$ and the family $\cU$ do not vary during phases but only between them.
We note here that we expose the edges of $\Aux_p$ during the exploration algorithm.

At the beginning of each phase, there is a \emph{clean up subroutine} which serves to maintain the expansion properties of the graph we are exploring. 
It consists in iteratively adding to $W$ vertices in $\Aux \setminus (U\cup W)$ of degree $C^{-1/14}d$ or less. 

After the clean up subroutine, there is an \emph{exploration phase}.
Every step in this phase consists of exposing the value of a random variable $X_e$, thus querying if the edge $e$ is in $\Aux_p$ or not.
At the start of each phase, we will have that $U_1=U_2 = \varnothing$ and each phase will run as long as $|U_1\cup U_2|<(C^{-1}+C^{-1/12})d$ and the number of queried edges is less than $2 C^{-13/12} d^2$.

During one phase, the set $U_2$ will have the role of a \emph{stack} where vertices added later are discarded earlier in the process (that is, $U_2$ follows a first-in-last-out rule). In particular, the most recent vertex added to $U_2$ will be called the vertex \emph{on top of the stack}.

In each step, consider the top vertex $v$ in the stack $U_2$ (if it exists). We then execute the first valid action among (a), (b) and (c) as long as possible, until we manage to execute (d) for once. 
\begin{itemize}
\item[(a)] If $v = V_P^{\pm}$ and $V_P^{\mp}$ is still in $Z$, move $V_P^{\mp}$ from $Z$ to $U_2$.
\item[(b)] If all the edges incident to $v$ have already been queried, then move $v$ from $U_2$ to $U_1$.
\item[(c)] If the stack is empty, we add the first vertex (with respect to $\sigma$ of $A\setminus (U\cup W\cup U_1)$ to $U_2$.
 \item[(d)] Otherwise, $v$ is incident to at least one unqueried edge. Consider the first such edge according to the ordering $\sigma_v$, say $vu$, and query it. If $X_{vu} = 1$, then move $u$ from $Z$ to $U_2$. 
\end{itemize}
Note that consecutive vertices in the stack are always adjacent in $\Aux_p\cup M$, so $U_2$ spans a path in this graph at any moment during the process. 
Furthermore, we note that, at every step, all edges between $U_1$ and $Z$ have been revealed to be absent in $\Aux_p$ at previous steps within the same phase of DFS-Aux.

At the end of every phase, the set $\{U_2\}$ is added to $\cU$ and all vertices in $U_1$ and $U_2$ are moved to $U$. 
The DFS-Aux algorithm terminates when $U\cup W = V(\Aux)$.

The vertices in $U\cup W$ should be thought of as \emph{forbidden} (for different reasons) and, in particular, once a vertex is added to $U\cup W$, DFS-Aux will not query edges incident to it in the future. The vertices in $U$ have been forbidden as we have already \emph{explored} them in the process of trying to construct our long paths, whereas the vertices in $W$ were forbidden during the clean up subroutine, as their degree became too low in the graph $\Aux\setminus (U \cup W)$ before the next exploration phase.

We now turn to the analysis of DFS-Aux. We will first show that, independently of the history of the process, in each phase, it is quite likely that a long path is added to $\cU$.

\begin{lemma}\label{l: DFS}
At the end of each phase of \textnormal{DFS-Aux}, for any outcome of the previous phases, with probability at least $1-\exp(-d/C)$, we have that  $|U_2|\geq  C^{-1/12}d$.
\end{lemma}
\begin{proof}
Recall that, by definition of the DFS-Aux algorithm, each phase ends when either the set $U_1\cup U_2$ reaches size $(C^{-1}+C^{-1/12})d$, or $2 C^{-13/12} d^2$ queries have been made. 

Let us first show that, at the end of the phase, $|U_1|< C^{-1}d$. Indeed, by the condition on the minimum degree of the graph $\Aux\setminus(U\cup W)$ ensured by updating the set $W$ at the beginning of every phase, and since $|U_1\cup U_2|\le (C^{-1}+C^{-1/12})d$, there must be at least
\begin{equation*}
    \left({d}{C^{-1/14}}-|U_1\cup U_2|\right)|U_1|\ge {d}{
    C^{-1/14}} |U_1|/2
\end{equation*}
edges between $U_1$ and $Z$ which have been queried at the current phase. As there are at most $2{C^{-13/12}}{d^2}$ queries in each phase, we have that $|U_1|< C^{-1}d$.

We now consider the two stopping conditions separately and show that, in each case, with probability at least $1-\exp(-d/C)$, we have that $|U_2|\ge C^{-1/12}d$. 
\vspace{0.5em}
\\\noindent
\textbf{Case 1.} Assume $|U_1\cup U_2| = (C^{-1}+C^{-1/12})d$. As $|U_1| \leq C^{-1}d$, we have $|U_2| \geq C^{-1/12}d$. 
\vspace{0.5em}
\\\noindent
\textbf{Case 2.} Assume $t = 2 C^{-13/12} d^2$ queries have been made. 
Every query is successful with probability $p$ and, for each successful query, one or two vertices are added to $U_1\cup U_2$: the unrevealed endpoint of the new edge and possibly its partner in $M$.

However, by the Chernoff bound, the number of successful queries during the current phase is smaller than $(C^{-1/12}+C^{-1})d$ with probability at most
\begin{align*}
\mathbb{P}\left(\mathrm{Bin}(t,C/d) \leq  (C^{-1/12} + C^{-1})d\right)\le \exp\left(-\frac{d}{C}\right). 
\end{align*}
If this event does not hold, then, since  $|U_1| \leq  C^{-1}d$, it follows that $|U_2|\geq  C^{-1/12}d$.
\end{proof}

With \cref{l: DFS} at hand, we can show that the paths in the family $\mathcal{U}$ produced by DFS-Aux after its termination cover almost all vertices in the final set $U$.

\begin{lemma}\label{l: DFS2}
At the end of \textnormal{DFS-Aux}, with probability $1-o(1/d)$, the set of paths in $\mathcal{U}$ of length at least $C^{-1/12}d$ spans at least $(1-2C^{-11/12})|U|$ vertices of $U$. 
\end{lemma}
\begin{proof}
Call each phase of DFS-Aux \textit{successful} if it produces a path of size at least $C^{-1/12}d$, that is, if at the end of the phase $|U_2|\geq C^{-1/12}d$. By \cref{l: DFS}, for any outcome of the previous phases, each phase is successful with probability at least $1-\exp(-d/C)$. Thus, the probability that more than $d^{-2}$ proportion of the phases are \textit{not} successful is at most
\begin{align}
\sum_{k\geq 1} 
\mathbb{P}(\text{DFS-Aux runs for exactly $k$ phases})\cdot \mathbb{P}(\mathrm{Bin}(k,\exp(-d/C))\geq k/d^2).    \label{eq: no-success}
\end{align}
By Markov's inequality,
\begin{align*}
    \mathbb{P}(\mathrm{Bin}(k,\exp(-d/C))\geq k/d^2)\le d^2\exp(-d/C)=o(1/d),
\end{align*}
and thus \eqref{eq: no-success} is at most
\begin{align*}
    o(1/d)\cdot \sum_{k\geq 1} 
\mathbb{P}(\text{DFS-Aux runs for exactly $k$ phases})= o(1/d)\cdot 1=o(1/d).
\end{align*}

If a phase is successful, then the vertices moved from $Z$ to $U$ in this phase which are not covered by a path of length at least $C^{-1/12}d$ in $\cU$ are those in $U_1$, and otherwise it is all vertices in $U_1 \cup U_2$. Due to the stopping condition $|U_1|+|U_2|\leq (C^{-1} +C^{-1/12})d$, in each successful phase $|U_1|\leq C^{-1}d$. 
Thus, denoting by $k$ the number of phases in the algorithm run, the proportion of vertices of $U$ that are not spanned by the paths in $\cU$ of length at least $C^{-1/12}d$ is at most
\[  
\frac{(k/d^2)(C^{-1} +C^{-1/12})d + k C^{-1}d}{\left(1-1/d^2\right) k C^{-1/12}d } = \frac{(C^{-1} +C^{-1/12})/d^2 +  C^{-1}}{\left(1-1/d^2\right)  C^{-1/12} } \leq 2C^{-11/12}.\qedhere
\]
\end{proof}

Recall that each vertex $v \in A$ corresponds to a segment in $L_{i,i+1}[V_1]$, and that each path in $\Aux \cup M$ corresponds to a way to merge some sequence of paths in $\cP_1(i)$ into a longer path covering most of the vertices in the original paths. 
Since the majority of the vertices in $L_{i,i+1}$ are contained in $V_1$, and the majority of vertices $v \in U$ are covered by long paths in $\cU$ constructed by DFS-$\Aux$, it remains to show that only a vanishingly small proportion of the vertices in $A$ are deleted during the clean up subroutines and added to $W$.

\begin{lemma}\label{l: peeling}
When \textnormal{DFS-Aux} ends, with probability $1-o(1/d)$, we have $|W\cap A|\le C^{-1/300}|A|$.
\end{lemma}
\begin{proof}
We abbreviate $W_A=W\cap A$ and $W_B=W\cap B$. For $X\in \{U, W_B, Z\}$, let $F_X$ be the set of edges in $\Aux$ incident to $W_A$ with the following property: at the moment when their endpoint in $A$ was added to $W$, their other endpoint was in $X$. By definition, the sets $F_U, F_{W_B}$ and $F_Z$ are pairwise disjoint. 

On the one hand, since $d_{\Aux}(v)\ge C^{1/10}d$ for every $v\in A$, there are at least $C^{1/10}d|W_A|$ edges between $W_A$ and $B$, and so 
\begin{equation}\label{eq:LBWA}
 |F_U| + |F_{W_B}| +|F_Z| \ge C^{1/10}d|W_A|.
\end{equation}
On the other hand, by the definition of $W$ and of the clean up subroutine, every time a vertex enters $W$ it has at most $C^{-1/14}d$ neighbours in $Z$. Thus,
\begin{align*}
|F_Z|\le C^{-1/14} d |W_A| \le C^{-1/14} d |A|. 
\end{align*} 
Similarly, for every edge $uv\in F_{W_B}$ with $u \in W_A$ and $v \in B$, we must have that $v$ entered $W$ before $u$ did, and so there was a point where $u \in Z$ and $v$ entered $W$. Hence, since at this point $v$ had at most $C^{-1/14}d$ neighbours in $Z$, it follows that
\begin{align*}
|F_{W_B}|\leq C^{-1/14}d|W_B|\le C^{-1/14}d|B|.
\end{align*} 
Finally let $uv\in F_U$ with $u\in W_A$ and $v\in B$. Thus, there is some phase in which $v$ moved to $U_2$, and during this phase $u$ was in $Z$. 

First, suppose that during this phase $v$ does not move from $U_2$ to $U_1$. Then $v$ belongs to some set in $\mathcal U$. Note that, for each set $P\in \mathcal U$, since $P$ spans a path in $\Aux_p$, at least a third of the vertices of $P$ lie in $A$.
In particular, in total, there are at most $2|A|$ vertices $v\in B\cap U$ belonging to some set in $\cU$. Thus, the number of edges in $F_U$ for which one of its endpoints belongs to some set in $\cU$ is at most 
$2|A| (\max_{v\in B} d_{\Aux}(v))\leq 2d|A|$.

Suppose that $v$ moves to $U_1$ during this phase. By \cref{l: DFS2}, with probability $1-o(1/d)$, there are at most $2C^{-11/12}|U| \leq 2C^{-11/12}(|A|+|B|)$ such vertices. Hence, with the same probability, there are at most 
\[2C^{-11/12}(|A|+|B|) (\max_{v\in B} d_{\Aux}(v))\leq 2C^{-11/12}d(|A|+|B|)\]
edges in $F_U$ containing a vertex in $B$ which entered $U_1$ at some point of the algorithm.

All in all, using \eqref{eq: A vs B}, we obtain that, with probability $1-o(1/d)$, 
\begin{align}\label{eq:barF}
|F_U| + |F_{W_B}| +|F_Z| &\le 2C^{-11/12}d(|A|+|B|) +  2d|A| +
C^{-1/14} d |B| +  C^{-1/14}d|A| \nonumber
\\&\le  3C^{-1/14}d|B|+3d|A| 
\leq 10C^{1/6-1/80-1/14}d|A|.
\end{align}

By combining~\eqref{eq:LBWA} and~\eqref{eq:barF}, we obtain that
\[C^{1/10}d|W_A|\le |F_U| + |F_{W_B}| +|F_Z| \le 10C^{1/6-1/80-1/14}d|A|.\]
Since $1/6-1/80 -1/14 - 1/10 < -1/300$, the desired conclusion follows.
\end{proof}

We are ready to prove Proposition~\ref{prop: long paths}.

\begin{proof}[Proof of Proposition~\ref{prop: long paths}] 
We condition on a successful outcome of Lemma~\ref{lem:paths_in_M}, and recall the resulting family of paths $\cP_1(i)$. Note that, in the proof of Lemma~\ref{lem:paths_in_M}, we only exposed the edges of $Q^d_p[L_{i,i+1}[V_1]]$. We construct the graph $\Aux$ as in the previous section, run the DFS-$\Aux$ algorithm and condition on a successful outcome of Lemmas~\ref{l: DFS2} and~\ref{l: peeling}. 
Note that these three events happen simultaneously with probability $1-o(1/d)$.

Let $\cU(i)$ be the family of paths guaranteed by \cref{l: DFS2}, and let $\cU'(i) \subseteq \cU(i)$ be the family of paths of length at least $C^{-1/12}d$. 
We refer to the paths in $\cU(i) \setminus \cU'(i)$ as \emph{short} paths. Moreover, let us call a path $R \in \cU'(i)$ \emph{bad} if $R$ contains at least $C^{-1/700}|R|$ or more vertices in $A$ whose partner in $M$ is not in $R$. We call the remaining paths \emph{good}.

Note that, by construction, for every edge in $\Aux$ between $v \in L_i[V_2]$ and $P^* \in \cS_1(i)$ there is at least one edge in $Q^d$ between $v$ and some vertex $u \in P^*$. Further, each edge of $\Aux$ is retained in $\Aux_p$ independently and with probability $p$. Hence, there is a natural coupling of $\Aux_p$ and $Q^d_p$ such that the edge $(v,P^*)$ is in $\Aux_p$ whenever $uv \in Q^d_p$. Hence, from each path $R \in \cU(i)$, we can construct a path $P(R)$ in $Q^d_p$ which contains at least a $(1-2C^{-1/8})$-proportion of the vertices in each path $P \in \cP_1(i)$ such that the edge $\{V_P^+,V_P^-\}$ of $M$ is contained in $R$ (see Figure~\ref{fig: DFS merge}). Furthermore, note that the paths constructed in this manner for different $R \in \cU(i)$ are disjoint.
\begin{figure}
\centering
\includegraphics[width=0.55\textwidth]{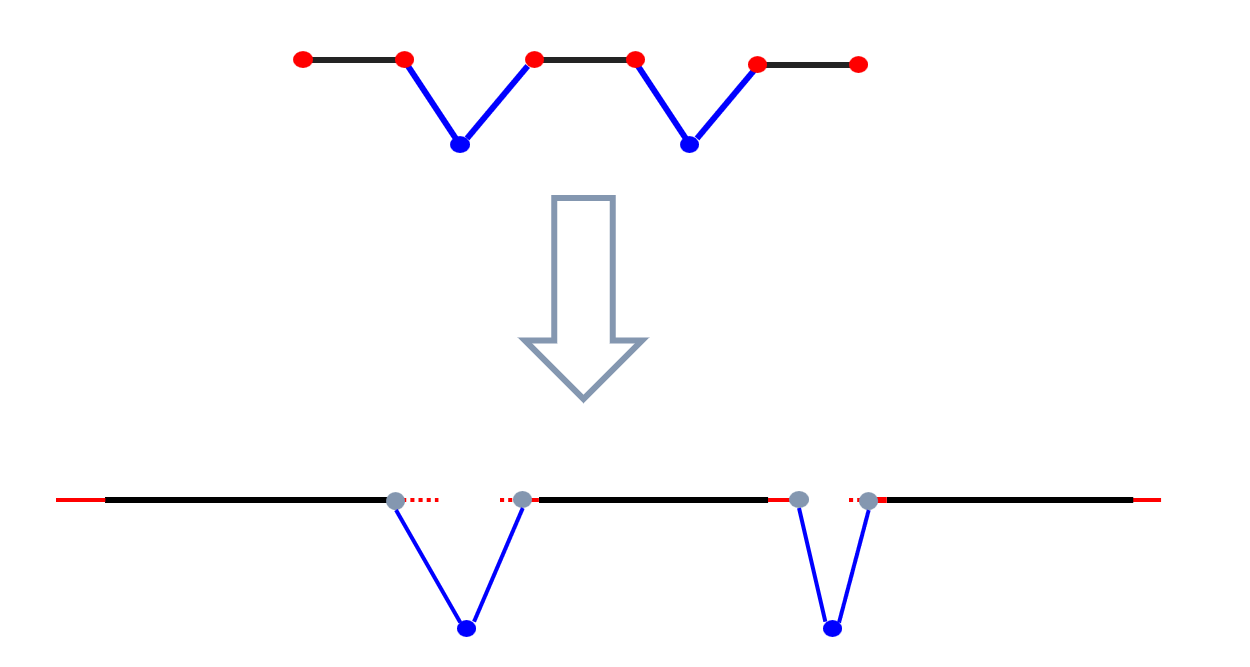}
\caption{In the upper part of the figure, we have a path found in DFS-Aux. The red circles are vertices in $A$, and the blue circles are vertices in $B$. The black edges correspond to edges in $M$. In the lower part of the figure, we have the corresponding path in $Q^d_p$. The paths appear in black, with their initial and terminal segments, which correspond to vertices of $A$, in red (and slightly thinner). The vertices in $B\subseteq V_2$ are in blue.} \label{fig: DFS merge}
\end{figure}

Now, by the stopping condition of DFS-Aux that $|U_1|+|U_2|\leq (C^{-1}+C^{-1/12})d$ 
and the bounds on the lengths of the paths in $\cP_1(i)$ from Lemma~\ref{lem:paths_in_M}, every path $R\in \cU(i)$ results in a path $P(R)$ of length at most 
\begin{equation}\label{eq:UBPR}
|P(R)|\le (C^{-1}+C^{-1/12})d\cdot 2C^{1/6} \le 4C^{1/12}d.
\end{equation}

On the other hand, for a good path $R\in \cU'(i)$, we have that at least a third of its edges lie in $M$. In particular,
\begin{align}\label{eq:LBPR}
|P(R)| \geq \frac{C^{-1/12}d}{3} \cdot (1-2C^{-1/8})\cdot C^{1/6} \geq \frac{1}{4}C^{1/12}d. 
\end{align}

Thus, all that is left is to bound the number of vertices in $L_{i,i+1}$ not covered by the family $\cP_2(i)\coloneqq \{P(R)\colon R\in \cU'(i), R\text{ is good}\}$. Such a vertex $v$ in $Q^d$ must satisfy one of the following.
\begin{enumerate}[(\arabic*){}]
   \item $v\notin \bigcup_{P\in \cP_1(i)}P$. Otherwise, let $P$ be the path that $v$ belongs to. \label{i: not in P1}
    \item $v$ is in $P^-$ or $P^+$ (and thus perhaps lost during some merging process).\label{i: in a segment}
    \item At least one of the segments $P^-$ and $P^+$ corresponds to some $S\in A$ in a short path.\label{i: in a short path}
    \item The segments $P^-,P^+$ correspond to some $S,S'\in A$ which are consecutive in a bad path $R\in \cU'(i)$.\label{i: bad path}
    \item The segments $P^-,P^+$ correspond to some $S,S'\in A$ which are not consecutive vertices on any path in $\cU'(i)$.\label{i: broken M}
\end{enumerate}

Note that if $v$ does not satisfy any of the above, then it must be covered by some $P(R)\in \cP_2(i)$. Let us estimate the number of vertices of each kind separately.

By \cref{def:spread partition} and Lemma~\ref{lem:paths_in_M}, there are at most $C^{-1/90}|L_{i,i+1}|$ vertices that are in $Q^d[L_{i,i+1}]$ but not in $\bigcup_{P\in \cP_1(i)}P$, that is, vertices of type~\ref{i: not in P1}. By construction, there are at most $\frac{C^{1/8}}{C^{1/6}}|L_{i,i+1}|=C^{-1/24}|L_{i,i+1}|$ of type~\ref{i: in a segment} in $Q^d[L_{i,i+1}]$. By Lemma~\ref{l: DFS2}, there are at most $2C^{-11/12}|U|$ vertices in short paths in our auxiliary graph. Thus, there are at most $2C^{1/6}\cdot 2C^{-11/12}|U|\le C^{-2/3}|L_{i,i+1}|$ vertices of type \ref{i: in a short path} in $Q^d[L_{i,i+1}]$. 
Moreover, by combining \Cref{l: peeling}, the fact that all paths in $\cP_1(i)$ have lengths in the interval $[C^{1/6},2C^{1/6}]$,~\eqref{eq:UBPR} and~\eqref{eq:LBPR}, the number of vertices of type~\ref{i: bad path} is bounded from above by $C^{-1/700} |L_{i,i+1}[V_1]|\le 2C^{-1/700} |L_{i,i+1}|$.
Finally, by similar considerations, the proportion of vertices of type \ref{i: broken M} is bounded from above by $2C^{-1/300}\cdot \frac{2C^{1/6}}{C^{1/6}} |L_{i,i+1}| \le C^{-1/400} |L_{i,i+1}|$.

Altogether, there are at most $C^{-1/800}|L_{i,i+1}|$ vertices of types \ref{i: not in P1}-\ref{i: broken M}, and the conclusion of the lemma holds.
\end{proof}

\section{Merging paths via growing trees}\label{sec:part 2}

We begin with a rough description of the proof strategy in this section. The aim of this section is to merge the paths in $ \cP_2 = \{ \cP_2(i) \colon i \in [m_3,m_4]\}$ (constructed in Proposition \ref{prop: long paths}) in two steps, each following a similar strategy --- we will grow expanding trees down, layer by layer, from some initial and terminal segments of each path in $\cP_2$. When a vertex enters two or more distinct trees of distinct paths, it is used to merge two distinct paths. 

In order to keep track of the merged path families and trees at each stage in this process, we introduce the following key definition.

\begin{definition}\label{def: PES}
Let $\cP$ and $\cS$ be families of vertex-disjoint paths in $Q^d_p$, where we call the elements of $\cS$ \emph{segments}, and let $\cF$ be a family of vertex-disjoint trees. We say the triple $(\cP,\cS,\cF)$ is a \emph{path-extension~forest} (PEF for short) if the following properties hold.
\begin{itemize}
    \item For every $P \in \cP$ there are two paths $P^+,P^-$ in $\cS$ which consist of the subpaths spanned by the first, resp. last, $C^{1/13}d$ vertices in $P$. Conversely, for every $S \in \cS$, there is a $P \in \cP$ and $* \in \{+,-\}$ such that $S = P^*$. We say that $S$ is a \emph{segment} of $P$.
    \item For every $S \in \cS$, there is a tree $T_S$ in $\cF$ which contains $S$. We say that $T_S$ \emph{extends} $S$. Conversely, for each $T \in \cF$, there are a unique path $P \in \cP$ and a segment $S \in \cS$ of $P$ such that $T \cap V(P) = S$.
\end{itemize}
\end{definition}

For a PEF $(\cP, \cS,\cF)$, we denote by $\textrm{Int}(\cP, \cS, \cF)$ the union over $P\in \cP$ of the set of vertices on the subpath of $P$ that connects $P^+$ and $P^-$, that is, $\textrm{Int}(\cP, \cS, \cF)=V(\cP)\setminus V(\cS)$. For readability, given $S = P^*$, we will write $T^*_P$ for $T_S$ and we will follow a similar convention for other notation in which a segment $P^*$ appears as a subscript.

\vspace{1em}

The aim of this section is to prove the following result.

\begin{proposition}\label{prop:part 2}
\Whp there exists a PEF $(\cP_3,\cS_3,\cF_3)$ in $Q^d_p$ satisfying all of the following properties.
\begin{enumerate}[\upshape{\textbf{A\arabic*}}]
    \item\label{item:C1} $V(\cP_3) \cup V(\cS_3) \cup V(\cF_3) \subseteq L_{m_1, m_2}[Q_0] 
    \cup L_{m_2+1,m_4+1}$.
    \item\label{item:C2} $|\cP_3|\le 6$.
    \item\label{item:C3} $|\textnormal{Int}(\cP_3,\cS_3,\cF_3)|\ge (1-C^{-1/850})2^d$. 
    \item\label{item:C4} For every $S\in \cS_3$, there is a set $W_S$ of $d^{20}$ leaves of $T_S$ in $L_{m_1}[Q_0]$ such that $J_S \coloneqq \mathbb{T}(W_S)$ has size at most $2m_1 = 100 \log d$. 
    In addition, the sets $W_S$ can be chosen so that $\{J_S: S \in \cS_3\}$  are pairwise disjoint. 
\end{enumerate}
\end{proposition}

Let us first describe the strategy on a high level. We begin by exposing the random partition (given by $(V_1, V_2, V_3)$) of $L_{m_4,m_4+1}$, which will typically be well-spread. We then apply \Cref{prop: long paths} to $Q^d_p[L_{m_1,m_1-1}[V_1\cup V_2]$ and initialise the PEF by taking $\cP=\cP_2(m_4)$, and letting $\cS=\cF$ be the segments of $\cP$. We stress that at this point, for every $S\in \cS$, $\cF\ni T_S=S$. Then, we will iteratively grow our PEF by exposing sequentially the edges of $Q^d_p$ between three consecutive layers $L_{i,i+2}$, together with the random (and, again, typically well-spread) partition of $L_{i,i+1}$.

First, we use the vertices in $L_{i,i+1}[V_3]$ to grow the trees in $\cF$ down from their vertices in $L_{i+2}$ to vertices in $L_i$. 
We say two trees $T_S,T_{S'} \in \cF$ are \emph{mergeable} if $S$ and $S'$ are segments of distinct paths in $\cP$. In this case, we also say that the two segments $S$ and $S'$ are mergeable.

More concretely, we expose the edges (in $Q^d_p$) of each $u\in L_{i+1}[V_3]$ to $L_{i+2}$ in some arbitrary order, and then that of each $v\in L_{i}[V_3]$ to $L_{i+1}$ in some arbitrary order. If a vertex $v \in V_3$ is revealed to be adjacent to a pair of mergeable trees $T,T'$, then we choose such a pair uniformly at random and \emph{merge them through $v$}. 
More precisely, if $P \neq R \in \cP$ are paths with segments $S,S'$ whose corresponding trees $T_S,T_{S'}$ merge, we replace $P,R$ with a new path $P'$ containing $P\setminus S,R\setminus S'$ as well as the path in $T_S\cup T_{S'}$ connecting $S,S'$ via the common vertex $v$. 
The segments of $P'$ are then the two segments of $P,R$ different from $S,S'$, and these segments maintain their associated trees in $\cF$ (see Figure \ref{f: MOG merge}).
If $v$ is revealed to be adjacent to a single tree, or two trees which are not mergeable, we choose one of the trees at random and add $v$ as a leaf to that tree.

Inside $L_{i,i+1}[V_1 \cup V_2]$, we apply \cref{prop: long paths} to construct the family $\cP_2(i)$ of paths. Then, we add this family of paths to $\cP$ and, for each $P \in \cP_2(i)$, we add its initial and terminal segments to $\cS$ and also to $\cF$.

\begin{figure}
\centering
\includegraphics[width=0.65\textwidth]{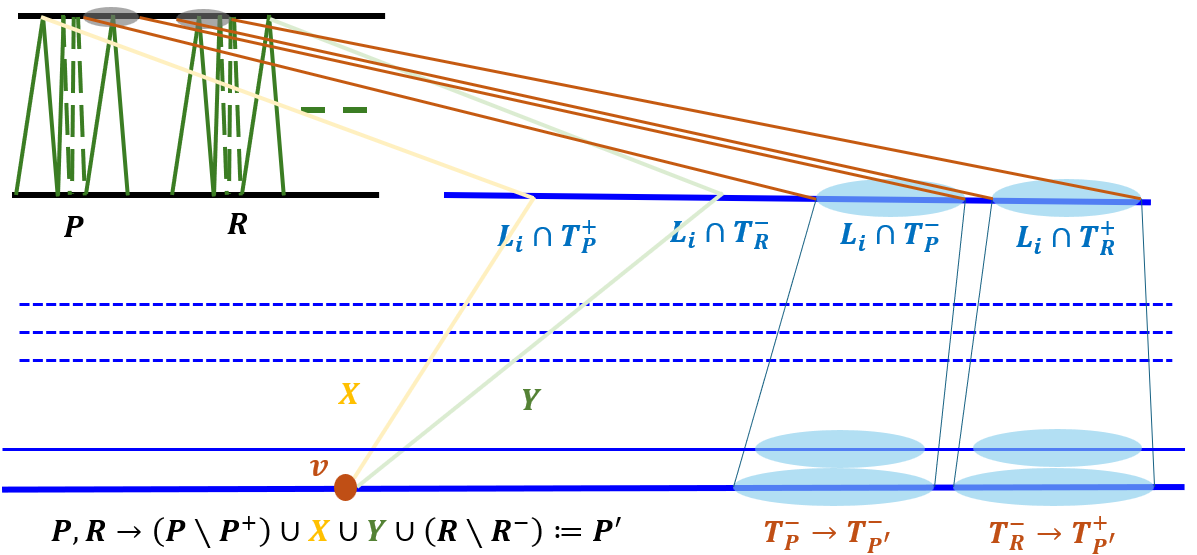}
\caption{Illustration of two paths, $P$ and $R$, merging through $v$ by using the path $X$ from the tree $T_{P}^+$ and the path $Y$ from the tree $T_{R}^-$. The new path is then $P'$, and we define $T_{P}^-,T_{R}^+$ to be its trees.} \label{f: MOG merge}
\end{figure}

Let us formally define then the algorithm described above, which we call the \emph{Merge-Or-Grow Algorithm} (MOG for short).

\paragraph{The Merge-Or-Grow Algorithm (MOG):} 

Set $(\cP^{m_4+2},\cS^{m_4+2},\cF^{m_4+2}) = (\varnothing,\varnothing,\varnothing)$. The Merge-Or-Grow Algorithm will run for a number of steps, one for each layer $L_i$ where $i \in [m_1,m_4]$ is even, in decreasing order. In the step corresponding to the $i$-th layer, we will have some PEF $(\cP^{i+2},\cS^{i+2},\cF^{i+2})$ which is contained in the layers $L_{i+2,d}$. We initialise $(\cP^{i+1},\cS^{i+1},\cF^{i+1})\coloneqq (\cP^{i+2},\cS^{i+2},\cF^{i+2})$. We note that the output of the algorithm will be a PEF $(\cP^i, \cS^i, \cF^i)$, and $(\cP^{i+1},\cS^{i+1},\cF^{i+1})$ serves as an intermediate step in the algorithm process.

We expose a random partition $(V_1,V_2,V_3)$ of $L_{i,i+1}$ by assigning each vertex to $V_j$ with probability $q_j$ independently, and we set $B_i \subseteq L_{i,i+1}$ to be
\begin{align*}
   B_i= \begin{cases}
         L_{i,i+1}[V_3] &\text{ if } i\in [m_2+1,m_4],
        \\  L_{i,i+1}[V_3\cap V(Q_0)] &\text{ if } i=m_2,
         \\ L_{i,i+1}[Q_0] &\text{ if } i\in [m_1,m_2-1].
    \end{cases}
\end{align*}
We now expose the edges of $Q^d_p[L_{i,i+2}]$ with at least one endpoint in $B_i$. We fix some arbitrary order of the vertices in $B_i$ where all vertices in layer $L_{i+1}$ appear before all vertices in layer $L_i$, and then we process the vertices in $B_i$ according to the following rules executed in the given order, updating the PEF $(\cP^{i+1},\cS^{i+1},\cF^{i+1})$ dynamically.
\begin{enumerate}[\upshape{\textbf{B\arabic*}}]
    \item\label{item:A1} Suppose $v\in B_i$ is adjacent in $Q^d_p$ to at least one pair of mergeable trees (that is, adjacent to mergeable segments or to leaves of mergeable trees). Then we select uniformly at random such a pair $T^{\bullet}_P$ and $T^{\circ}_R$ where $P \neq R \in \cP^{i+1}$ and $\bullet, \circ\in \{+,-\}$. We denote by $X$ the path in $Q^d_p$ from $v$ to $P\setminus P^{\bullet}$ which goes through $T^{\bullet}_P$, and by $Y$ the path in $Q^d_p$ from $v$ to $R\setminus R^{\circ}$ which goes through $T^{\circ}_R$. We delete the paths $P$ and $R$ from $\cP^{i+1}$, the segments $P^{\bullet}$ and $R^{\circ}$ from $\cS^{i+1}$ and the trees $T^{\bullet}_P$ and $T^{\circ}_R$ from $\cF^{i+1}$. We add a new path $P'\coloneqq (P\setminus P^{\bullet})\cup X\cup Y\cup(R\setminus R^{\circ})$ to $\cP^{i+1}$ (see Figure \ref{f: MOG merge}). The segments of $P'$ are taken to be $P^{\star}$ and $R^{\diamond}$ where $\star=\{+,-\}\setminus\bullet,\diamond=\{+,-\}\setminus\circ$, and they maintain their associated trees in $\cF^{i+1}$.
    \item\label{item:A2} If $v\in B_i$ is adjacent in $Q^d_p$ to exactly two trees in $\cF^{i+1}$ which are not mergeable, then we choose one of these trees $T$ uniformly at random (with probability $1/2$), and add $v$ to $T$ as a leaf.
    \item\label{item:A3} If $v$ is adjacent in $Q^d_p$ to exactly one tree $T$ in $\cF^{i+1}$, then we add $v$ to $T$ as a leaf.
\end{enumerate}
We note that, if $v$ is an isolated vertex, we naturally do nothing and move to the next vertex. Further, we note that if a tree has no neighbours in $B_i$ and was not used for a merge, we discard the tree and its associated path and segment; however, as we will soon see, \textbf{whp} this event never occurs. We will say that a vertex $v \in B_i$ is \emph{processed according to \ref{item:A1}-\ref{item:A3}} during the iteration of MOG on $L_{i,i+1}$ according to which one of the above rules is applied to $v$. It will sometimes be convenient to refer to the status of the PEF after the vertices in $B_i \cap L_{i+1}$ have been processed, but before the vertices in $B_i \cap L_i$ have been processed. We will refer to this as the point in time \emph{after the first layer has been processed} in the iteration of MOG on $L_{i,i+1}$.

Finally, if $i \geq m_3$, then we expose the edges of $Q^d_p[L_{i,i+1}[V_1 \cup V_2]]$ and, conditionally on the success of \cref{prop: long paths}, we add the paths in $\cP_2(i)$ to $\cP^{i+1}$, and add their initial and terminal segments of length $C^{1/13}d$ to both $\cS^{i+1}$ and $\cF^{i+1}$. Also, we set  $(\cP^{i},\cS^{i},\cF^{i})=(\cP^{i+1},\cS^{i+1},\cF^{i+1})$ to be the resulting PEF.

We note a few properties of MOG that will be useful in our later analysis.

Firstly, by \cref{r:wellspread}, for each even $i \in [m_3,m_4]$, the random partition of $L_{i,i+1}$ generated during the iteration of MOG in $L_{i,i+1}$ is well-spread with probability $1-o(1/d)$. Hence, by a union bound, \whp the application of \cref{prop: long paths} is successful for each even $i \in [m_3,m_4]$. 
In particular, after the first iteration of MOG in $L_{m_4,m_4+1}$, we have that, in the PEF $(\cP^{m_4},\cS^{m_4},\cF^{m_4})$, $\cP^{m_4}$ is the path family $\cP_2(m_4)$ generated by \cref{prop: long paths}.

Secondly, for even $i,j$ with $m_1 \leq i < j \leq m_4$ and a tree $T_j \in \cF^{j}$, we stress that exactly one of the following two scenarios occurs.
\begin{itemize}
    \item There exists an even $k$ with $i \leq k < j$ and a unique tree $T_k\in \cF^k$ such that $T_k$ belongs to some pair of trees which are merged (via \ref{item:A1}) during the $k$-th iteration of MOG, and $T_j$ is a subtree of $T_k$. In this case, we say that $T_j$ \emph{was utilised in a merge}.
    \item There exists a unique tree $T_{i}\in \cF^{i}$ such that $T_j$ is a subgraph of $T_{i}$ (i.e. $T_{i}$ is the extension of $T_j$ in $\cF_{i}$). In this case, we say $T_{i}$ \emph{extends} $T_j$. 
\end{itemize}

We will also require the following technical lemma which says that, if a vertex $v \in B_i$ was processed according to \ref{item:A1} during the iteration of MOG on $L_{i,i+1}$ and $v$ was revealed to be adjacent to some tree $T \in \cF^{i+2}$, then it is not \emph{too} unlikely that $T$ was one of the trees that was merged through $v$. 
Indeed, intuitively, a typical vertex $v$ is adjacent to at most $(1+o_C(1))C/2$ other vertices on the previous layer, and thus at most $(1+o_C(1))C/2$ trees in $\cF^{i+2}$. 
Excluding some pairs of edges incident to $v$ and connecting non-mergeable pairs of trees, each pair is equally likely to be merged. Hence, we should expect that a single tree is merged through $v$ with probability roughly at most $1/C$.

\begin{lemma}\label{lem:MOGmergeprob}
Fix even $i \in [m_1,m_4]$ and $v \in B_i$. Consider the point at which $v$ is processed during the iteration of MOG on $L_{i,i+1}$. 
Conditionally on the events that $v$ is used to merge two trees in the iteration of MOG on $L_{i,i+1}$ and that $v$ is revealed to be adjacent to $T \in \cF^{i+2}$, the probability that $v$ is used to merge $T$ to another tree is at least $1/(3C)$.
\end{lemma}
\begin{proof}
Let $\mathcal{R}$ be the set of trees in $\cF^{i+2}$ which $v$ is revealed to be adjacent to in $Q^d_p$ during MOG, where $T \in \mathcal{R}$ by assumption. Let $T'$ be the unique tree in $\cF^{i+2}$ not mergeable with $T$ and let us define $\ell \coloneqq | \mathcal{R} \setminus \{T,T'\}|$. Note that, under the conditioning in the lemma, $\ell$ is stochastically dominated by $\mathrm{Bin}(d,p)+1$. (Indeed, while the affect of the conditioning that $\ell\ge1$ implicitly depends on the tree $\hat T\in(\mathcal{R}\setminus\{T,T'\})$ which we condition $v$ to be adjacent to, the probability that $\ell \le 2C$ is at least $1-\exp(-C/4)$ for \textit{any} choice of a tree $\hat T$ for which we condition).

In particular, by a standard Chernoff-type bound, the probability that $\ell\le 2C$ is at least $1-\exp(-C^2/4C)=1-\exp(-C/4)$. On the other hand, conditioned on $\ell$ there are at most $\binom{\ell+2}{2}$ mergeable pairs of trees adjacent to $v$, and at least $\ell$ of these pairs contain $T$. Hence, under the conditioning in the lemma, the probability that $v$ is used to merge $T$ with another tree is at least
\begin{align}\label{eq: prob ordinary}
    \mathbb{P}\left(\mathrm{\ell\le 2C}\right)\cdot\min_{\ell\le 2C}\frac{\ell}{\binom{\ell+2}{2}}\ge \frac{1}{3C}.
\end{align}
\end{proof}

Rather than proving Proposition \ref{prop:part 2} directly, we will first analyse an intermediary step in the process and consider the PEF $(\cP^{m_2},\cS^{m_2},\cF^{m_2})$ after the iteration of MOG on $L_{m_2,m_2+1}$. Note that the paths in $\cP^{m_2}$ lie in $L_{m_2,m_4+1}$ by construction. We will show that \textbf{whp} each of the trees in $\cF^{m_2}$ has a large set of leaves in $L_{m_2}[Q_0]$ with large joint support --- the latter property will be useful in our analysis of MOG on layers $L_{m_1,m_2}$, where the degree to the next layer decreases significantly.

\begin{proposition}\label{prop:2(a)}
Let $(\cP^{m_2},\cS^{m_2},\cF^{m_2})$ be the PEF obtained by running MOG until after the iteration on $L_{m_2,m_2+1}$. Then, \whp the following hold.
\begin{itemize}
\item $|\textnormal{Int}(\cP^{m_2},\cS^{m_2},\cF^{m_2})|\geq (1-C^{-1/850})2^d$.
\item For every $S \in \cS^{m_2}$, there exist a set of leaves $M_S \subseteq V(T_S)\cap L_{m_2}[Q_0]$ of size $d^{C^{3/4}}$ and a set $I_S \subseteq \ind(M_S)$ of size $|I_S| = d/6$.
\end{itemize}
\end{proposition}
Note that, when $i < m_2$, no new paths are added in the iteration of MOG on $L_{i,i+1}$. 
In particular, the sets of segments $\cS^{i}$ form a non-increasing nested sequence as $i\in [m_1,m_2]$ decreases.

We then analyse the remaining iterations of MOG until after the iteration on $L_{m_1,m_1+1}$ and show that typically, during this period, for each segment $S\in \cS^{m_2}$, either the corresponding tree $T_S \in \cF^{i}$ is used to merge two paths at some point, 
or there are many disjoint paths within the corresponding tree $T_S \in \cF^{m_1}$ ending in layer $L_{m_1}$. Furthermore, we will show that, for segments $S$ and $S' \in \cS^{m_2}$ whose corresponding coordinate sets $I_S$ and $I_{S'}$ from \Cref{prop:2(a)} have a `large' intersection, it is very unlikely that both segments `survive' and remain in $\cS^{m_1}$, unless they end up as the two segments corresponding to some path $P \in \cP^{m_1}$.

\begin{proposition}\label{prop:2(b)}
Let $(\cP^{m_1},\cS^{m_1},\cF^{m_1})$ be the PEF obtained by running MOG until after the iteration on $L_{m_1,m_1+1}$, and define the sets $I_S$ for $S\in \cS^{m_2}$ as in \Cref{prop:2(a)}. Then, \whp\hspace{-0.5mm}, for every distinct $S,S'\in \cS^{m_2}$, if $|I_S \cap I_{S'}| \geq d/10^3$, then exactly one of the following holds.
\begin{enumerate}[\em (a)]
    \item \label{i:usedmerge} At least one of the segments $S,S'$ is not in $\cS^{m_1}$, that is, at least one of the trees $T_S,T_{S'}\in \cF^{m_2}$ was utilised in a merge during some iteration of MOG.
    \
    \item \label{i:samepath} There exists a path $P \in \cP^{m_1}$ with segments $P^+ =S$ and $P^-=S'$.
\end{enumerate}
In addition, \whp\hspace{-1.5mm}, for every $S \in \cS^{m_1} \subseteq \cS^{m_2}$ with a corresponding tree $T_S \in \cF^{m_1}$, and every subset $I \subseteq I_S$ of size $100 \log d = 2m_1$, there is a set $Z_I \subseteq V(T_S) \cap L_{m_1}$ such that $|Z_I| = d^{20}$ and $\mathbb{T}(Z_I) \subseteq I$.
\end{proposition}
We note that the set $I_S$ in the above proposition will be the same as in Proposition~\ref{prop:2(a)}.

The proofs of Propositions~\ref{prop:2(a)} and~\ref{prop:2(b)} may be found in Sections~\ref{sec:2(a)} and~\ref{sec:2(b)}, respectively. We now deduce Proposition~\ref{prop:part 2} from Propositions~\ref{prop:2(a)} and~\ref{prop:2(b)}.

\begin{proof}[\textbf{Proof of Proposition~\ref{prop:part 2} assuming Propositions~\ref{prop:2(a)} and~\ref{prop:2(b)}.}]
In what follows, we 
condition on the events given by Propositions~\ref{prop:2(a)} and~\ref{prop:2(b)}, which occur \whp\hspace{-1.5mm}. We note that $|\interior(\cP^{m_2},\cS^{m_2},\cF^{m_2})|\ge (1-C^{-1/850})2^d$. 
Since the interior of the PEFs $(\cP^i,\cS^i,\cF^i)$ is non-decreasing by construction, it follows that \ref{item:C3} holds. Furthermore, \ref{item:C1} follows from the definition of the set $B_i$ in the description of MOG, as the PEF $(\cP^{m_2},\cS^{m_2},\cF^{m_2})$ is contained in $L_{m_2,m_4+1}$ and, for every even $i \in [m_1,m_2]$, $B_i \subseteq V(Q_0)$.

We now proceed to prove \ref{item:C2}. Assume that $|\cP^{m_1} |\geq 7$ and let $P_1,\ldots,P_7\in \cP^{m_1}$. Let $S_1,\ldots,S_7\in \cS^{m_1} \subseteq \cS^{m_2}$ be such that $S_i$ is the segment $P_i^+$ and let $I_{S_1}, \ldots, I_{S_7} \subseteq [d]$ be the sets of coordinates given by \cref{prop:2(a)}. Then, among these seven sets, there are two sets whose intersection is of size at least $d/10^3$. Indeed, by inclusion-exclusion,
\begin{align*}
    d\ge \left|\bigcup_{i=1}^{7}I_{S_i}\right|\ge \sum_{i=1}^7|I_{S_i}|-\sum_{i<j}|I_i\cap I_j|=\frac{7d}{6}-\sum_{i<j}|I_i\cap I_j|.
\end{align*}
Thus, there is some pair $(i,j)$ for which $|I_i\cap I_j|\ge\frac{d/6}{\binom{7}{2}}\ge \frac{d}{10^3}$. Since $S_i$ and $S_j$ satisfy neither \cref{prop:2(b)}\ref{i:usedmerge} nor \cref{prop:2(b)}\ref{i:samepath}, this contradicts  \cref{prop:2(b)}. Hence, $|\cP^{m_1} |\le 6$ and \ref{item:C2} holds.

Finally, we prove that \whp \ref{item:C4} holds. Let $\cS^{m_1}=\{S_1,...,S_{2s}\}$ for some $s\leq 6$, and let $I_{S_1},\ldots, I_{S_{2s}}$ be as in \cref{prop:2(a)}, where each $I_{S_j}$ has size $d/6$. 
For $k \leq 2s$, we define recursively $J_{k}$ to be an arbitrary subset of $I_{S_k} \setminus (\bigcup_{l=1}^{k-1} J_{l})$ or size $2m_1$, which exists since $m_1 = o(d)$. For each $k \leq 2s$, we let $W_{S_k}$ be the set $Z_{J_k}$ guaranteed by \cref{prop:2(b)}. Then, each $W_{S_k}$ has size $d^{20}$, $\mathbb{T}(W_{S_k}) \subseteq J_k$ has size at most $2m_1$, and the disjointness of the set of indices $\mathbb{T}(W_{S_k})$ follows from the disjointness of the $J_{k}$, as desired.
\end{proof}

\subsection{\texorpdfstring{Growing trees down to layer $L_{m_2}$}{}}\label{sec:2(a)}

This section is dedicated to the proof of \Cref{prop:2(a)}. Roughly speaking, we will show that \textbf{whp}, at each iteration of MOG, every tree in $\cF$ will either be used to merge two paths in $\cP$, or its set of leaves will grow at least by a factor of (roughly) $C^{1/4}$ until it reaches size $d^{C^{3/4}}$, and maintain this (or larger) size thereafter.

When a path first enters the PEF, we will analyse the first exploration downwards in $V_3$ from its corresponding trees slightly differently compared to the subsequent steps, as this step is crucial for finding an (eventually) suitable choice for the coordinate set in \ref{item:C4} of Proposition~\ref{prop:part 2}. 

The next lemma will be used to control how a tree grows in a general iteration of MOG (and will also be of use when analysing the first `merge-or-grow' step).

\begin{lemma}\label{lem:generalstep}
Fix even $i \in [m_2,m_4]$ and the PEF $(\cP^{i+2},\cS^{i+2},\cF^{i+2})$ obtained by running MOG until after the iteration on $L_{i+2,i+3}$. 
Further, fix a tree $T \in \cF^{i+2}$, reveal the set $N_{Q^d_p}(T) \cap B_i \cap L_{i+1}$, and assume $K\subseteq N_{Q^d_p}(T) \cap B_i \cap L_{i+1}$ satisfies $|K|\ge 7Cd$. Then, with probability $1-o(2^{-d}/d)$, one of the following holds.
\begin{itemize}
    \item $T$ is used for a merge during the iteration of MOG on $L_{i,i+1}$.
    \item At least $|K|/5$ of the vertices in $K$ are attached to $T$ (in particular by processing these vertices according to \ref{item:A2}-\ref{item:A3}).
\end{itemize}

Similarly, if $T \in \cF^{i+1}$ after the first layer has been processed and $K \subseteq N_{Q^d_p}(T) \cap B_i \cap L_{i}$ has size at least $7Cd$, then, with probability $1-o(2^{-d}/d)$, either $T$ is used for a merge during the iteration of MOG on $L_{i,i+1}$, or at least $|K|/5$ of the vertices in $K$ are attached to $T$.
\end{lemma}
\begin{proof}
Note that, by assumption, every vertex in $K$ is processed according to one of \ref{item:A1}-\ref{item:A3}.
We split the analysis into two cases. 

Firstly, suppose that at least $|K|/2$ of the vertices in $K$ are processed according to \ref{item:A1}. Then, by \cref{lem:MOGmergeprob}, the probability that $T$ is not used for a merge during this iteration of MOG is at most
\[
\left(1-(3C)^{-1}\right)^{|K|/2}\le \exp\left(-{|K|/(6C)}\right)\le \exp\left(-7d/6\right) =o(2^{-d}/d),
\]
where we used the fact that the event that $T$ was not merged through $v$ is independent for different $v\in K$.

Conversely, if at least $|K|/2$ of the vertices in $K$ were processed according to one of \ref{item:A2} or \ref{item:A3}, then the number of vertices in $K$ attached to $T$ by MOG stochastically dominates $\mathrm{Bin}(|K|/2,1/2)$. A standard Chernoff bound shows that this number is at least $|K|/5$ with probability $1-o(2^{-d}/d)$. 
The second part of the statement is proved analogously.
\end{proof}

With \cref{lem:generalstep} in hand, we analyse how the trees in $\cF$ grow during the first iteration of MOG in which they are processed.

\begin{lemma}\label{lem:step 1}
Fix even $i\in [m_3,m_4]$ and the family of paths $\cP_2(i+2)$ added to the PEF during the iteration of MOG on $L_{i+2,i+3}$. For any $P \in \cP_2(i+2)$ and for any segment $S \in \cS^{i+2}$ of $P$, with probability $1-o(2^{-d}/d)$, one of the following holds.

\begin{enumerate}[\em (a)]
    \item\label{i:merge} $T_S$ is used for a merge at the iteration of MOG on $L_{i,i+1}$.
    \item\label{i:grow} Else, let $T_S \in \cF^{i}$ be the tree corresponding to $S$. Then, there exist a set $I_S \subseteq \{2,\ldots,d\}$ of size $d/6$ and a set $M_S \subseteq V(T_S) \cap L_i$ of size $|M_S| \geq C^{7/4}d$ such that $I_S \subseteq \ind(M_S)$.
\end{enumerate}
\end{lemma}

 \begin{proof}
 Note that, since $P$ was added in the previous iteration of MOG, the tree $T_S$ is equal to the segment $S$, which is a subpath of $P$ of length $C^{1/13}d$. Hence, if we write $K' = V(T_S) \cap L_{i+2}$, then $|K'| = C^{1/13}d/2$. Then, by \cref{thm:LKK,l:binomtech}, 
 \[
 |N_{Q^d}(K') \cap L_{i+1}| \geq i|K'|/3 \geq d|K'|/7.
 \]
Recall that, during one iteration of MOG, we expose the edges of $Q^d_p$ as well as the partition $(V_1,V_2,V_3)$ of the new layers (where a vertex lands in $V_3$ with probability $q_3= C^{-1/80}$). 
By setting $K \coloneq N_{Q^d_p}(K') \cap L_{i+1}[V_3]$, a Chernoff bound implies that, with probability $1-o(2^{-d}/d)$, $|K| \geq C^{79/80}|K'|/10$.

Thereafter, \cref{lem:generalstep} implies that, with probability $1-o(2^{-d}/d)$, either $T_S$ is used for a merge in the iteration of MOG on $L_{i,i+1}$, or MOG attaches to $T_S$ at least $C^{79/80}|K'|/10\cdot 1/5$ vertices from the set $L_{i+1}[V_3]$. We assume the latter, as otherwise \emph{\ref{i:merge}} holds.

Let us divide the segment $S$ into $t=C^{1/10+1/13}$ paths of equal length and let $K'_1,\ldots, K'_t$ be their intersections with $L_{i+2}$, so that $K' = \bigcup_{j=1}^t K'_j$ and $|K'_j| = C^{-1/10}d$ for all $j$. Then, by the pigeonhole principle, there is some $j \in [t]$ such that
\[
|N_{Q^d_p}(K'_j) \cap L_{i+1}[V_3]| \geq C^{79/80 - 1/10 - 1/13} |K'|/50 \geq 2 C^{7/8} d.
\]Let us write $K_j \coloneqq N_{Q^d_p}(K'_j) \cap L_{i+1}[V_3]$.

Since the vertices in $K'_j$ are joined in $Q^d$ by a path of length at most $C^{-1/10}d$, all such vertices agree on at least $(1-C^{-1/10})d$ coordinates. Furthermore, all vertices in $K'_j$ lie in $L_{i+2}$ and thus have support of size at least $m_3$. It follows that $|\ind(K'_j)| \geq m_3 - C^{-1/10}d \geq d/3+2$.

Let us split $|\ind(K'_j)|$ into two sets, $J_1$ and $J_2$, of (nearly) equal size. 
Again, by the pigeonhole principle, without loss of generality, at least $|K_j|/2 \geq C^{7/8}d$ of the vertices in $K_j$ differ from its neighbour in $K'_j$ by a coordinate outside $J_2$.

Let $M'_S \subseteq K_j$ be a subset of size $|M'_S| = C^{7/8}d$ with $\mathds{1}(M'_S)\subseteq J_2$.
Then, it follows that $|\ind(M'_S)| \geq d/6 +1$, and so there exists some subset $I_S \subseteq \ind(M'_S) \setminus \{1\}$ of size $d/6$.

We now run a similar argument to extend the vertices in $M'_S$ down to $L_i$ without changing the coordinates in $I_S$. To do so, we consider the subcube $Q' \subseteq Q^d$ formed by taking all vertices $v$ such that $I_S \subseteq \ind(v)$ which, in particular, contains $M'_S$. 
Then, $Q'$ has dimension $5d/6$ and so, by another application of \cref{l:binomtech}, we have that
\[
|N_{Q'}(M'_S) \cap L_i| \geq d|M'_S|/7. 
\]
Hence, by the Chernoff bound, with probability $1-o(2^{-d}/d)$, writing $M_S\coloneqq N_{Q'_p}(M'_S)\cap L_{i+1}[V_3]$ we have
\[
|M_S| \geq  C^{79/80}|M'_S|/10 \geq C^{7/4} d,
\]
and $I_S \subseteq \ind(M_S)$ by construction. 
\end{proof}

Finally, using Lemmas~\ref{lem:generalstep} and~\ref{lem:step 1}, we can prove Proposition~\ref{prop:2(a)}.

\begin{proof}[Proof of Proposition~\ref{prop:2(a)}]
For every segment $S$ present at some point of MOG, denote by $i(S)$ the largest even $i$ such that $S\in \cS^i$. Note that $i(S)\ge m_3$ by the MOG description. 
We show inductively that, for every even $j\in [m_2,m_4-2]$ and $S \in \cS^{j}$ with $i(S) > j$, there is a subset $I_S \subseteq \{2,\ldots, d\}$ of size $d/6$ such that, with probability $1-o((i(S)-j)2^{-d}/d)$, the corresponding tree $T_S \in \cF^{j}$ is such that there is a subset $M_j^S \subseteq V(T_S) \cap L_j$ with $I_S \subseteq \ind(M_j^S)$ of size
\[
|M_j^S| \geq \min\{  C^{(12 + i(S)-j)/8}d, d^{C^{3/4}+5} \}.
\]
Given this inductive argument, when $j=m_2$, for every $S\in \cS^j$, the above inequality holds with probability $1-o((i(S)-m_2)2^{-d}/d)=1-o(2^{-d})$. A union bound over the at most $2^d/d$ segments and at most $d$ choices for $j$ completes the proof. We thus turn now to the proof of the inductive statement.

The base case $j=m_4-2$ is satisfied with probability $1- o(2^{-d}/d)$ with $I_S$ and $M_{m_4-2}^S = M_S$ as given by \cref{lem:step 1}, for all $S\in \cS^{m_2}$.

Fix even $j\in [m_2,m_4-4]$ and suppose that the claim holds for all even $j'\in [j+2,m_4-2]$. Fix $S \in \cS^{j+2}$ and a set $M_{j+2}^S \subseteq V(T_S) \cap L_{j+2}$ of size
\[
|M_{j+2}^S| \geq \min\{  C^{(10 + i(S)- j)/4}d, d^{C^{3/4}+5}\}
\]
with $I_S \subseteq \ind(M_{j+2})$. If $|M_{j+2}^S| \geq d^{C^{3/4}+5}$, then choose an arbitrary $K \subseteq M_{j+2}^S$ of size precisely $d^{C^{3/4}+5}$, otherwise let $K=M_{j+2}^S$.

We split the rest of the analysis into two cases according to whether $j=m_2$ or not.

\hspace{3mm}
\\ \noindent \textbf{Case 1: $j> m_2$.} 
As in the proof of \cref{lem:step 1}, we consider the subcube $Q'$ consisting of vertices $v$ with $I_S \subseteq \ind(v)$, which has dimension $5d/6$. Note that, in this subcube, $K$ lies in the $(j + 2 - d/6)$-th layer.

Since $|K| \leq d^{C^{3/4}+5}$, by \cref{thm:LKK,l:binomtech} it follows that 
\[
|N_{Q'}(K) \cap L_{j+1}| \geq \frac{|K| (j + 2 - d/6)}{2C^{3/4}+10} \geq \frac{d |K|}{8C^{3/4}}. 
\]
Subsequently, since $|K| \geq C^{(10 + i(S)- j)/4}d\ge C^{10/4}d$, by the Chernoff bound, with probability $1-o(2^{-d}/d)$,
\[
|N_{(Q')_p}(K) \cap L_{j+1}[V_3]| \geq \frac{1}{2} \cdot
 \frac{C}{d} \cdot q_3 \cdot \frac{d |K|}{8C^{3/4}} \geq 5C^{1/8}|K|.
\]

Hence, by Lemma \ref{lem:generalstep}, either $T_S$ is used in a merge in the iteration of MOG on $L_{j,j+1}$ and thus $S\notin \cS^j$, or, with probability $1-o(2^{-d}/d)$, after processing the vertices in $B_j \cap L_{j+1}$, 
\[
|V(T_S) \cap L_{j+1}[V_3] \cap V(Q')| \geq C^{1/8}|K|.
\]

By an analogous argument, after we have processed the vertices in $B_j \cap L_j$, we can conclude that either $S\notin \cS^j$ or
\[
|V(T_S) \cap L_{j}[V_3] \cap V(Q')| \geq C^{1/4}|K| \geq  \min\{  C^{(12 + i(S)-j)/8}d, d^{C^{3/4} +5} \}.
\]
Then, by choosing $M_j^S$ to be a subset of $V(T_S) \cap L_{j}[V_3] \cap V(Q')$ of the appropriate size, we have $I_S \subseteq M_j^S$ since $M_j^S \subseteq V(Q')$. 

\hspace{3mm}
\\ \noindent \textbf{Case 2: $j=m_2$.}
Fix a segment $S$ as before and note that
\[
i(S)-j=i(S)-m_2 \geq m_3 - m_2 \geq d^{0.6} \geq C^{3/4} \log d ,
\]
and so the induction hypothesis implies that $|K|= d^{C^{3/4}+5}$. 

Furthermore, note that $|N_{Q'}(K) \cap V(Q_0)| \geq |K|/d$, since every vertex in $K$ has at least one neighbour in $L_{m_2} \cap V(Q_0)$, and the degree of $Q^d$ is $d$.

Arguing by a similar argument as before with \cref{thm:LKK,l:binomtech} together with the Chernoff bound, we can conclude that either $T_S$ is used for a merge during the iteration on MOG on $L_{m_2}$, or the tree $T_S \in \cF^{m_2}$ is such that, with probability $1-o(2^{-d}/d)$,
\[
|V(T_S) \cap L_{m_2}[V_3] \cap V(Q') \cap V(Q_0)| \geq p \cdot q_3 \cdot |K| / d \geq d^{C^{3/4}},
\]
and choosing $M_{m_2}^S$ to be any subset of these vertices of size exactly $d^{C^{3/4}}$ satisfies the conclusion of \cref{prop:2(a)}. 
In both cases, the probability of failure is $o((i(S)-(j+2))2^{-d}/d)+o(2^{-d}/d)=o((i(S)-j)2^{-d}/d)$, as required.
\end{proof}

\subsection{\texorpdfstring{Connecting almost all paths in $\cP^{m_2}$ in the subcube $Q_0$}{}}\label{sec:2(b)}

We now turn to the proof of Proposition~\ref{prop:2(b)}. For this, we will analyse how the trees in $\cF^{m_2}$ continue to grow as we iterate MOG down from $L_{m_2}$ to $L_{m_1}$. 
To this end, we will introduce a simpler process where the growth of trees is stochastically dominated by the growth of the random trees in MOG. 
We call it the \textit{alternative growth procedure} (AGP for short). The AGP grows graphs in a mixed-percolated hypercube $(Q_0)_p(1/2)$.

By coupling the AGP with the trees grown using MOG from a fixed pair of segments, we show that the conclusion of \cref{prop:2(b)} holds with high enough probability to deduce the result via a union bound.

\paragraph{Alternative growth procedure for a graph $T$:}

Fix connected graphs $T_1,T_2 \subseteq L_{m_2,d}$ and a mixed-percolated hypercube $(Q_0)_p(1/2)$. 
Define $H$ to be the restriction of $(Q_0)_p(1/2)$ to $L_{m_1,m_2}$, and let $K \subseteq L_{m_1,m_2}$. 
Then, we write $\agp(T_i,K)$ for the connected subgraph of $(H \cup T_i) \setminus K$ induced by the vertices which can be reached from $T$, by following monotone decreasing paths in $H$. Similarly, we write $\agp(T_1,T_2,K) = \agp(T_1,K) \cup \agp(T_2,K)$. When $K = \varnothing$, we will simply write $\agp(T) = \agp(T,\varnothing)$ and $\agp(T_1,T_2)=\agp(T_1,T_2,\varnothing)$.

\vspace{1em}

We first motivate why coupling AGP with the trees grown using MOG is, in fact, quite natural. 
Note that, if $T\in \cF^{m_2}$ and $T$ was not used in a merge during MOG, then, at each iteration of MOG, every vertex which was revealed to be adjacent in $Q^d_p$ to $T$ and was processed according to \ref{item:A2} or \ref{item:A3} is appended to the tree $T$ with probability at least $1/2$. 
Hence, putting aside vertices adjacent in $Q^d_p$ to $T$ and processed according to \ref{item:A1} (of which, since $T$ was not used in a merge, by \Cref{lem:MOGmergeprob} there are typically few), 
one can see that the tree in $\cF^{m_1}$ which extends $T$ will dominate the graph obtained by gradually attaching each neighbour of $T$ in $Q^d_p$ to $T$ with probability $1/2$.
At the same time, AGP is actually significantly easier to analyse than MOG due to the presence of only one or two trees in the process.

More formally, the connection between AGP and MOG is exhibited by the following coupling lemma.

\begin{lemma}\label{lem:domination}
Fix a mergeable pair of segments $S,S' \in \cS^{m_2}$ with corresponding trees $T_S,T_{S'} \in \cF^{m_2}$. Then, there exists a coupling $\cC$ of the AGP (that is, a pair $((Q_0)_p(1/2),K)$) with the iterations of MOG on the layers in $L_{m_1,m_2-1}$ such that, under the coupling $\cC$, exactly one of the following holds. 
\begin{enumerate}[label=\emph{(\alph*)}]
    \item
    \label{i:onemerged} At least one of $S,S'$ is not in $\cS^{m_1}$, that is, at some point a tree corresponding to $S$ or $S'$ is used for a merge.
    \item
    \label{i:notmergeable} $S,S' \in \cS^{m_1}$ are not a mergeable pair, that is, there is some $P \in \cP^{m_1}$ such that $\{P^+,P^-\} = \{S,S'\}$.
    \item
    \label{i:prune}  $S,S' \in \cS^{m_1}$ are a mergeable pair and there exists a set $K \subseteq L_{m_1,m_2-1}$ such that,
    for the corresponding trees $\hat{T}_S,\hat{T}_{S'} \in \cF^{m_1}$, $V(\agp(T_S,K))\subseteq V(\hat{T}_S)$ and 
$V(\agp(T_{S'},K))\subseteq V(\hat{T}_{S'})$.   
Furthermore, there exists an event $\cE=\cE(S,S')$ which holds with probability $1-o(4^{-d})$ such that, conditionally on $\cE$, the set $K$ has size at most $d^3$.
\end{enumerate}
\end{lemma}
\begin{proof}
Let $(Z_v)_{v\in V(Q_0)}$ be a sequence of independent Bernoulli$(1/2)$ random variables.
We run the iterations of MOG on the layers $L_{m_1,m_2-1}$, where, whenever a vertex $v$ is processed according to \ref{item:A2} and is adjacent to the tree extending $T_S$ but not the tree extending $T_{S'}$, then we append it to $T_S$ if and only if $Z_v=1$. Let us write $\hat{K}(S)$ for the set of vertices processed in this manner. Similarly, if $v$ is processed according to \ref{item:A2} and is adjacent to the tree extending $T_{S'}$ but not the tree extending $T_{S}$, then we append it to $T_{S'}$ if and only if $Z_v=1$. Let us write $\hat{K}(S')$ for the set of vertices processed in this manner. 

Denote by $K$ the set of vertices which are processed by MOG according to \ref{item:A1} and are adjacent to the (growing) trees extending $T_S$ or $T_{S'}$. We couple $\agp(T_S,T_{S'},K)$ with the process described above by identifying edge-percolation in $H$ with edge-percolation in $Q^d$ processed by MOG, and by including each vertex in $\hat{K}(S)\cup \hat{K}(S')$ in $(Q_0)_p(1/2)$ if $Z_v=1$, and all other vertices according to Bernoulli($1/2$) random variables, independent from the current construction. This completes the description of our coupling $\cC$.

Clearly, it cannot happen that two among \emph{\ref{i:onemerged}}, \emph{\ref{i:notmergeable}} and \emph{\ref{i:prune}} simultaneously occur. 
In the sequel, we suppose that neither \emph{\ref{i:onemerged}} nor \emph{\ref{i:notmergeable}} occurs --- if some of them take place, we abandon the coupling automatically.
In this case, we show that \emph{\ref{i:prune}} occurs. 

We claim inductively that, at the iteration of MOG on $L_{j,j+1}$ for an even $j \in [m_1,m_2]$, the set of vertices in the tree extending $T_S$ contains $L_{j,j+1}[V(\agp(T_S, K))]$. Indeed, the claim clearly holds for $j=m_2$. 
For the induction step, any neighbour $v$ in $Q^d_p$ of $T_S$ processed according to \ref{item:A1} belongs to $K$, and therefore does not belong to $V(\agp(T_S, K))$. For neighbours processed according to \ref{item:A2}, since we assumed that \emph{\ref{i:onemerged}} and \emph{\ref{i:notmergeable}} do not occur, at this moment $S$ and $S'$ are still a mergeable pair. 
Hence, if $v$ is processed according to \ref{item:A2}, it is not adjacent to $T_{S'}$, and so $v \in \hat{K}(S)$ is attached to $T_S$ if and only if $Z_v =1$. Hence, if $v$ is contained in $\agp(T_S,K)$, then $Z_v =1$ and $v$ is attached to $T_S$ at this iteration of MOG. A similar argument holds for $S'$.

It remains to prove that $|K|\leq d^3$ with probability $1-o(4^{-d})$. Indeed, the event $|K| > d^3$ implies that there exists an even $j\in [m_1,m_2]$ such that, at the iteration of MOG on $L_{j,j+1}$, some of these layers contain at least $2d^3/(m_2 -m_1) \geq 2d^2$ vertices which are processed according to \ref{item:A1} and are adjacent to the trees extending $T_S$ and $T_{S'}$. 
Hence, by Lemma \ref{lem:MOGmergeprob}, the probability that neither of these vertices is used to merge the extension of $T_S$ or the extension of $T_{S'}$ by MOG is bounded from above by $(1-1/3C)^{2d^2} = o(4^{-d})$. Thus, $|K|\ge d^3$ and \emph{\ref{i:prune}} jointly occur with probability $o(4^{-d})$, as desired.
\end{proof}

\begin{remark}\label{rem:single coupling}
A similar argument shows that, for any $S \in \cS^{m_2}$, there is a coupling of AGP with the iterations of MOG on the layers in $L_{m_1,m_2-1}$ such that, with probability $1-o(4^{-d})$, either $S \not\in \cS^{m_1}$, or $|K|\leq d^3$ and $V(\agp(T_S,K))\subseteq V(\hat{T}_S)$ where $\hat{T}_S \in \cF^{m_1}$ is the tree corresponding to $S$.
\end{remark}

The next lemma is the last tool needed to deduce \cref{prop:2(b)}. Roughly, given trees $S,S' \in \cF^{m_2}$ with $|I_S \cap I_{S'}| \geq d/10^3$, it uses the self-similarity properties of the hypercube to guarantee that we can find `many' disjoint paths from $T_{S}$, resp. $T_{S'}$, to $L_{m_1}$ in $\agp(T_S)$, resp. $\agp(T_{S'})$. However, if there are more than $|K|$ of these paths, then $\agp(T_S,K)$ and $\agp(T_{S'},K)$ must share vertices in $L_{m_1}$, which will exclude the third case of \cref{lem:domination}.

\begin{lemma}\label{lem:paths in AGP}
Let $T_1,T_2\subseteq L_{m_2,d}$ be trees (not necessarily distinct). Suppose that for each $T\in \{T_1,T_2\}$, there exist a set $M_T\subseteq L_{m_2}$ of leaves of $T$  of size $d^{C^{3/4}}$, and a set $I_T\subseteq [d]$  of coordinates of size $d/6$ such that $I_T\subseteq \ind(M_T)$. 
In addition, assume that $|I_{T_1}\cap I_{T_2}|\geq d/10^3$.
Then the following holds with probability $1-o(4^{-d})$. 

For every $I\subseteq I_{T_1} \cap  I_{T_2}$ with $|I|=100\log d=2m_1$, there is a set $Z_I \subseteq L_{m_1}[Q_0]$ with $|Z_I|=2d^{20}$ which is contained in both $\agp(T_1)$ and $\agp(T_2)$, considered as subgraphs of $\agp(T_1,T_2)$, such that $\mathbb{T}(Z_I) \subseteq I$. Furthermore, each $z\in Z_I$ is connected in $\agp(T_1,T_2)$ to $T_1$ via a pair of internally vertex-disjoint monotone paths, and similarly to $T_2$ via a pair of internally vertex-disjoint monotone paths; furthermore, these paths are disjoint for distinct $z$.
\end{lemma}

We will prove \cref{lem:paths in AGP} in the next subsection, let us first show that \cref{prop:2(b)} follows from \cref{lem:paths in AGP}

\begin{proof}[Proof of Proposition~\ref{prop:2(b)} assuming \cref{lem:paths in AGP}]
Let us condition on the statement of Proposition~\ref{prop:2(a)}. 
Fix $S,S'\in \cS^{m_2}$ and let $I_S$ and $I_{S'}$ be given by Proposition~\ref{prop:2(a)} and assume that $|I_S \cap I_{S'}| \geq d/10^3$. Finally, let $I \subseteq I_S \cap I_{S'}$ have  size $100\log d$. 

By Lemma \ref{lem:domination}, we can couple the iterations of MOG between $L_{m_2}$ and $L_{m_1}$ with a pair $((Q_0)_p(1/2),K)$ such that exactly one of the following holds with probability $1-o(4^{-d})$.
\begin{itemize} 
    \item At least one of $S,S'$ is not in $\cS^{m_1}$.
    \item $S,S' \in \cS^{m_1}$  are not a mergeable pair.  
    \item $S,S' \in \cS^{m_1}$ and $K \subseteq L_{m_1,m_2}$ is such that $|K|\leq d^3$, $V(\agp(T_S,K)) \subseteq V(\hat{T}_S)$ and $V(\agp(T_{S'},K)) \subseteq V(\hat{T}_{S'})$ where $\hat{T}_S,\hat{T}_{S'} \in \cF^{m_1}$ are the trees corresponding to $S$ and $S'$. 
\end{itemize}
We thus only need to rule out  the third case. Note that, in this case, since $S,S' \in \cS^{m_1}$ are distinct, $V(\hat{T}_S)$ and $V(\hat{T}_{S'})$ are disjoint.

However, by \cref{lem:paths in AGP}, with probability $1- o(4^{-d})$ there is a set $Z_I \subseteq L_{m_1}[Q_0]$ of size $2d^{20}$ such that each vertex in $z$ is joined by a pair of internally vertex-disjoint monotone paths in $\agp(T_S,T_{S'})$ to $T_S$ and $T_{S'}$, and these paths are disjoint for distinct $z$. Since $|K|\leq d^3$, at least $2d^{20} - d^3$ of the vertices $z \in Z_I$ are joined to both $T_S$ and $T_{S'}$ by a monotone path in $\agp(T_S,T_{S'},K)$ and hence lie in $V(\agp(T_S,K)) \cap V(\agp(T_{S'},K)) \subseteq V(\hat{T}_S) \cap V(\hat{T}_{S'})$, contradicting the fact that they are disjoint.

Taking a union bound over the at most $4^d$ possible pairs of segments $S \in \cS^{m_2}$ completes the proof of the first part of \cref{prop:2(b)}.

For the second part of \cref{prop:2(b)}, given $S \in S^{m_1}$, we apply \cref{rem:single coupling} to $S \in \cS^{m_2}$ and \cref{lem:paths in AGP} to the pair $T_S,T_S \in \cF^{m_2}$ and some subset $I \subseteq I_S$ of size $|I| = 2m_1$, and an argument similar to the above shows the existence of an appropriate set of leaves $Z_I \subseteq L_{m_1}$ with $|Z_I| = d^{20}$ and $\mathbb{T}(Z_I) \subseteq I$ with probability $1-o(4^{-d})$. Again, a union bound over the at most $2^d$ segments $S \in \cS^{m_2}$ completes the proof.
\end{proof}

\subsection{\texorpdfstring{Growing trees in $(Q_0)_p(1/2)$}{}}
In this section, we prove \cref{lem:paths in AGP}. Roughly the idea will be to use \cref{lem:firstsubcubes} to find, for each relevant $I$, a large set of disjoint subcubes whose top and bottom points lie in $M_{T_i}$ and $L_{m_1}$, respectively, and then to use \cref{lem: monotone paths} to show that, for many of these subcubes, there is a monotone path from $L_{m_1}$ to $M_{T_i}$ in the AGP.

For technical reasons, we do this in two steps: first, we connect vertices in $M_{T_i}$ to an intermediate layer in the described manner, and then connect the intermediate vertices with~$L_{m_1}$.

\begin{proof}[Proof of \cref{lem:paths in AGP}]
Let $T_1$, $T_2$, $M_{T_1}$, $M_{T_2}$, $I_{T_1}$ and $I_{T_2}$ be as in the statement of \cref{lem:paths in AGP}. Let $M = M_{T_1} \cup M_{T_2}$, let $I :=I_{T_1} \cap I_{T_2}$ and let $I' \subseteq I$ be a subset of size $|I'| = 2m_1$. Note that, by the assumptions of \cref{lem:paths in AGP}, $I \subseteq \ind(M)$ and $d/10^3 \leq |I| \leq d/6$. 

Let $Z = \{ v \in L_{m_1} \colon \ind(z) \subseteq I'\}$. Then, $\mathbb{T}(Z) = I' \subseteq I$ and 
\[|Z|=\binom{|I'|}{m_1}=\binom{2m_1}{m_1}= 2^{m_1+o(m_1)}=d^{50+o(1)}.\] 
By \Cref{lem:firstsubcubes}, there exists a set $\{u(z,y) \colon y\in M, z\in Z\} \subseteq L_{m_1+m_2 - |I|}$ such that 
\begin{enumerate}[\textnormal{(\arabic*)}]
    \item\label{i:top2} for each $z \in Z$, the subcubes $\{Q[u(z,y);y] \colon y \in M\}$ are pairwise disjoint, and
    \item\label{i:bottom2} for each $z \neq z'$ and each pair $y,y' \in M$ (not necessarily distinct), the two subcubes $Q[z;u(z,y)]$ and $Q[z';u(z',y')]$ are disjoint.
\end{enumerate}

Let $(Q_0)_p(1/2)$ be used to generate $\agp(T_1,T_2)$. Recall that, for $i\in \{1,2\}$, every monotone path in $(Q_0)_p(1/2)$ from $M_{T_i}$ which ends in $L_{m_1}$ is included in $\agp(T_i)$.

Since each subcube $Q[u(z,y);y]$ has dimension $|I|-m_1 \geq d/10^4$ and $C$ is sufficiently large, by Lemma~\ref{lem: monotone paths}, each $Q[u(z,y);y]$ contains a monotone path in $(Q_0)_p(1/2)$ from $y$ to $u(z,y)$ with probability $\Omega(d^{-5})$. 
Moreover, by \ref{i:top2}, these events are independent for different $y\in M$. Therefore, since $|M_{T_1}|=|M_{T_2}|=d^{C^{3/4}}$, the Chernoff bound implies that, for every $z\in Z$, with probability $1-o(2^{-3d})$, 
there exist $y^1_z \in M_{T_1}$ and $y^2_z\in M_{T_2}$ such that monotone paths from $y^1_z$ to $u(z,y^1_z)$ and $y^2_z$ to $u(z,y^2_z)$ are realised within their respective hypercubes in $(Q_0)_p(1/2)$. 

Note that the vertices $y^i_z$, $u(z,y^i_z)$ and the monotone paths from $y^i_z$ to $u(z,y^i_z)$ in $Q[u(z,y^i_z);y^i_z]$ for $i\in \{1,2\}$ are determined by the set of edges (and vertices) of $(Q_0)_p(1/2)$ spanned by $L_{m_1+m_2 -|I|,d}$.

In what follows, we suppose that we have exposed these edges and identified the vertices $y^i_z$ and $u(z,y^i_z)$, and the monotone paths from $y^i_z$ to $u(z,y^i_z)$ in $Q[u(z,y^i_z);y^i_z]$ for $i\in \{1,2\}$. 
Given $z \in Z$, let us consider the subcubes $Q[z;u(z,y^1_z)]$ and $Q[z;u(z,y^2_z)]$. Note that these subcubes are of dimension $m_2 - |I|\ge  d/4$ and that, since both subcubes are contained within $L_{m_1,m_1+m_2-|I|}$, the edges in these subcubes have not yet been exposed. Hence, by Harris' inequality (see, e.g., \cite[Theorem 6.3.3]{AS16}) and Lemma~\ref{lem: monotone paths}, the two increasing events that $Q[z;u(z,y^1_z)]$ contains a monotone path from $z$ to $u(z,y^1_z)$ and that $Q[z;u(z,y^2_z)]$ contains a monotone path from $z$ to $u(z,y^2_z)$ hold jointly with probability $\Omega(d^{-10})$.

Furthermore, by \ref{i:bottom2}, for each $z \neq z'$, we have that $Q[z;u(z,y^1_z)] \cup Q[z;u(z,y^2_z)]$ is disjoint from $Q[z';u(z',y^1_{z'})] \cup Q[z';u(z',y^2_{z'})]$. 
Hence, the events of existence of the monotone paths from $z$ to $u(z,y^i_z)$ for $i\in \{1,2\}$ are independent for different $z \in Z$ and, by Chernoff's inequality, with probability $1-o(2^{-3d})$, there is a set $Z_I$ of at least $|Z|/d^{11}\ge 2d^{20}$ such that, for each $z\in Z_I$, the monotone paths from $z$ to $u(z,y^i_z)$ for $i\in \{1,2\}$ are contained in $(Q_0)_p(1/2)\cup T_1\cup T_2$.

Finally, by a union bound over the at most $\binom{d/6}{2m_1}\leq 2^{4m_1^2}$ subsets of $I$ of size $2m_1$, we deduce that, with probability $1-o(4^{-d})$, for every choice of a set $I'$, there exist a set $Z_{I'} \subseteq Z$ of at least $2d^{20}$ vertices such that, for each $z\in Z_{I'}$, there exists paths from $z$ to $u(z,y^i_z)$ and from $u(z,y^i_z)$ to $y^i_z \in M_{T_i}$ for $i\in \{1,2\}$ disjoint for distinct choices of $z$, thanks to \ref{i:top2} and \ref{i:bottom2}. By concatenating these paths, we obtain $|Z_{I'}|$ disjoint paths from $Z_{I'}$ to $T_1$ and from $Z_{I'}$ to $T_2$, completing the proof.
\end{proof} 

\section{Stitching the remaining paths into a cycle}\label{sec: part 3}
We now deduce \cref{thm:main} from \cref{prop:part 2}. 
\begin{proof}[Proof of \cref{thm:main}]
We start by exposing the edges of $H_1 = Q^d_p[L_{m_1,m_2}[Q_0]\cup L_{m_2+1,m_4+1}]$. By~\cref{prop:part 2}, we have that \whp $H_1$ spans a PEF $(\cP_3,\cS_3, \cF_3)$ that satisfies~\ref{item:C2}-\ref{item:C4}. We condition on the existence of such a PEF, and we write $\cP_3 = \{P_1,P_2,\ldots,P_s\}$ for some $s\leq 6$. 

By \ref{item:C4} together with Definition \ref{def: PES}, for every $P\in \cP_3$ and $*\in\{+,-\}$, there exists a set $W_P^*\subseteq T_P^*\cap L_{m_1}[Q_0]$ with $|W_P^*|\ge d^{20}$ such that $J_{P}^*=\mathbb{T}(W_P^*)$ is of size at most $100\log d$. Moreover, by \ref{item:C4}, we may assume that $\{J_{P}^*\colon P\in\cP_3, *\in \{+,-\}\}$ are pairwise disjoint. Let $J\coloneqq \bigcup_{P\in \cP_3, *\in \{+,-\}}J_P^*$, and note that $|J|\le 1200\log d$. 

We will show that, after exposing the remaining edges of $Q^d_p$, \whp\hspace{-1.5mm}, for every $r\in [s]$, there exists a path whose internal vertices lie outside of $L_{m_1,m_2}[Q_0]\cup L_{m_2+1,m_4+1}$, and its endpoints lie in $W_{P_r}^-$ and $W_{P_{r+1}}^+$ (indices seen modulo $s$), and that these paths are vertex disjoint for each $r \in [s]$. Given such a family of paths, it is straightforward to deduce the existence of a cycle in $Q^d_p$ spanning $\interior(\cP_3,\cS_3,\cF_3)$ and, by \ref{item:C3}, such a cycle has length at least $|\interior(\cP_3,\cS_3,\cF_3)|\geq (1-C^{-1/850})2^d$.

Fix an arbitrary set $K\subseteq \{2,\ldots,d\}\setminus J$ of size $2d/3$ and, for each $k \in K$, let $u_k\in L_2$ and $v_k\in L_{d/3+1}$ be the vertices such that $\ind(u_k) = \{1,k\}$ and $\ind(v_k) = [d] \setminus (K \setminus \{k\})$. Note that, for every $k\in K$, we have that $J\subseteq \ind(v_k)$. Then, by \cref{lem:secondsubcube}, the subcubes $Q[u_k;v_k]$ are pairwise vertex-disjoint for different $k\in K$, are subcubes of $Q_1$ of dimension $d/3-1$, and are contained in $L_{1,d/3+1}$. In particular, for each $k \in K$, the edges in $Q[u_k;v_k]$ are yet to be exposed. 

Let us take an arbitrary balanced partition of $K$ into $s$ subsets $K_1,\ldots,K_s$. Then, for every $r \in [s], k \in K_r$ and $u \in W^-_{P_r} \cup W^+_{P_{r+1}}$, let $P(k,u)$ be the path $uu_1u_2$ of length two which starts at $u$, traverses an edge along the first coordinate, and ends with an edge along the $k$-th coordinate, that is, $\ind(u_1)=\ind(u)\cup\{1\}$, and $\ind(u_2)=\ind(u)\cup \{1,k\}$. Note that $k\notin\ind(u)$ since $J\cap K=\varnothing$, and $u_2\in Q[u_k;v_k]$.

We note that the family $\{P(k,u) \colon r \in [s], k \in K_r, u \in W^-_{P_r} \cup W^+_{P_{r+1}}\}$ consists of vertex-disjoint paths with edges yet to be exposed.
Furthermore, these paths share no edges with the subcubes $\{Q[u_k;v_k] \colon k \in K\}$. 
Fix $H_2=Q^d\setminus\left(H_1\cup \{Q[u_k;v_k] \colon k \in K\}\right)$.
Then, each $P(k,u)$ is contained in $H_2$ independently with probability $p^2$.
Thus, since there are at most $\Theta(d)$ pairs $r \in [s]$ and $k \in K_r$, and $|W^-_{P_r}| = |W^+_{P_{r+1}}| = d^{20}$, by the Chernoff bound and a union bound, \whp\hspace{-1.5mm}, for every $r \in [s]$ and $k \in K_r$, there exist vertices $x_{r,k}^-\in W_{P_r}^-$ and $x_{r,k}^+\in W_{P_{r+1}}^+$ such that the paths $P(k,x_{r,k}^*)$ are contained in $H_2$.

For each $r\in [s], k\in K_r$ and $*\in\{+,-\}$, we denote the other endpoint of the path $P(k,x_{r,k}^*)$ by $y_{r,k}^*$ and recall that it lies in $Q[u_k;v_k]$. 
We now expose the edges of $Q^d_p$ inside the subcubes $\{Q[u_k;v_k] \colon k \in K\}$. Since $Q[u_k;v_k]$ has dimension $d/3 -1$, the constant $C$ can be chosen sufficiently large so that the probability a vertex in $Q[u_k;v_k]$ is in the largest component is at least $1/2$ (see, e.g., \cite{AKS81,BKL92}). Hence, by Harris' inequality (see, e.g., \cite[Theorem 6.3.3]{AS16}), for each $r \in [s]$ and $k \in K_r$,
\begin{align*}
\mathbb{P}( y_{r,k}^- &\text{ and } y_{r,k}^+ \text{ are connected in } Q[u_k;v_k]_p) \\&\geq \mathbb{P}( y_{r,k}^- \text{ and } y_{r,k}^+ \text{ lie in the largest component of } Q[u_k;v_k]_p)\ge 1/4.
\end{align*}
Since the subcubes $\{Q[u_k;v_k] \colon k \in K\}$ are vertex-disjoint, these events are independent for distinct pairs $r$ and $k$. Hence, by the Chernoff bound, \whp\hspace{-1.5mm}, for each $r \in [s]$, there exists $k_r \in K_r$ such that $y_{r,k_r}^-$ and $y_{r,k_r}^+$ are connected in $Q[u_k;v_k]_p$ by a path $P'_r$. 

For each $r \in [s]$, the path formed by $P(k_r,x_{r,k_r}^-) \cup P'_r \cup P(k_r,x_{r,k_r}^+)$ joins $W^-_{P_r}$ and $W^+_{P_{r+1}}$ outside of $H_1$. These paths are vertex-disjoint for distinct $r \in [s]$ by construction, which finishes the proof.
\end{proof}

\section*{Acknowledgement}
This research was funded in part by the Austrian Science Fund (FWF) [10.55776/P36131, 10.55776/F1002, 10.55776/ESP624, 10.55776/ESP3863424], and by NSF-BSF grant 2023688. Part of this work was conducted while the second author was visiting IST Austria and TU Graz, and while the first, fourth and fifth authors were visiting the Simons Laufer Mathematical Sciences Institute, within the programs {\em Probability and Statistics of Discrete Structures} and {\em Extremal Combinatorics}. The authors thank these institutions for hospitality.  For open access purposes, the authors have applied a CC BY public copyright license to any author accepted manuscript version arising from this submission.

\bibliographystyle{abbrv}
\bibliography{main} 

\begin{thebibliography}{10}

\bibitem{AKS81a}
M.~{Ajtai}, J.~{Koml\'{o}s}, and E.~{Szemer\'{e}di}.
\newblock {The longest path in a random graph}.
\newblock {\em {Combinatorica}}, 1:1--12, 1981.

\bibitem{AKS81}
M.~Ajtai, J.~Koml\'{o}s, and E.~Szemer\'{e}di.
\newblock Largest random component of a {$k$}-cube.
\newblock {\em Combinatorica}, 2(1):1--7, 1982.

\bibitem{AKS85}
M.~Ajtai, J.~Koml{\'o}s, and E.~Szemer\'{e}di.
\newblock First occurrence of {H}amilton cycles in random graphs.
\newblock {\em North-Holland Mathematics Studies}, 115(C):173--178, 1985.

\bibitem{AS16}
N.~Alon and J.~H. Spencer.
\newblock {\em The probabilistic method}.
\newblock Wiley Series in Discrete Mathematics and Optimization. John Wiley \& Sons, Inc., Hoboken, NJ, fourth edition, 2016.

\bibitem{ADEK24}
M.~Anastos, S.~Diskin, D.~Elboim, and M.~Krivelevich.
\newblock Climbing up a random subgraph of the hypercube.
\newblock {\em Electron. Commun. Probab.}, 29:13, 2024.
\newblock Id/No 70.

\bibitem{B84}
B.~{Bollob\'as}.
\newblock {The evolution of random graphs}.
\newblock {\em {Trans. Am. Math. Soc.}}, 286:257--274, 1984.

\bibitem{B90}
B.~Bollob\'{a}s.
\newblock Complete matchings in random subgraphs of the cube.
\newblock {\em Random Structures Algorithms}, 1(1):95--104, 1990.

\bibitem{BKL92}
B.~Bollob\'{a}s, Y.~Kohayakawa, and T.~{\L}uczak.
\newblock The evolution of random subgraphs of the cube.
\newblock {\em Random Structures Algorithms}, 3(1):55--90, 1992.

\bibitem{B77}
J.~D. Burtin.
\newblock The probability of connectedness of a random subgraph of an {$n$}-dimensional cube.
\newblock {\em Problemy Pereda\v{c}i Informacii}, 13(2):90--95, 1977.

\bibitem{CDE24}
M.~Collares, J.~Doolittle, and J.~Erde.
\newblock The evolution of the permutahedron.
\newblock {\em arXiv preprint arXiv:2404.17260}, 2024.

\bibitem{CDGKO21}
P.~Condon, A.~Espuny~D\'iaz, A.~Gir\~ao, D.~K\"uhn, and D.~Osthus.
\newblock Hamiltonicity of random subgraphs of the hypercube.
\newblock {\em Mem. Amer. Math. Soc.}, 304(1534):v+132, 2024.

\bibitem{DEKK24}
S.~Diskin, J.~Erde, M.~Kang, and M.~Krivelevich.
\newblock Isoperimetric inequalities and supercritical percolation on high-dimensional graphs.
\newblock {\em Combinatorica}, 44(4):741--784, 2024.

\bibitem{DEKK22}
S.~Diskin, J.~Erde, M.~Kang, and M.~Krivelevich.
\newblock Percolation on high-dimensional product graphs.
\newblock {\em Random Structures \& Algorithms}, 66(1):e21268, 2025.

\bibitem{EKK22}
J.~Erde, M.~Kang, and M.~Krivelevich.
\newblock Expansion in supercritical random subgraphs of the hypercube and its consequences.
\newblock {\em Ann. Probab.}, 51:127--156, 2023.

\bibitem{ER60}
P.~Erd\H{o}s and A.~R\'{e}nyi.
\newblock On the evolution of random graphs.
\newblock {\em Magyar Tud. Akad. Mat. Kutat\'{o} Int. K\"{o}zl.}, 5:17--61, 1960.

\bibitem{ES79}
P.~Erd\H{o}s and J.~Spencer.
\newblock Evolution of the {$n$}-cube.
\newblock {\em Comput. Math. Appl.}, 5(1):33--39, 1979.

\bibitem{Fdlv79}
W.~Fernandez de~la Vega.
\newblock Long paths in random graphs.
\newblock {\em Studia Sci. Math. Hungar.}, 14(4):335--340, 1979.

\bibitem{F86}
A.~Frieze.
\newblock On large matchings and cycles in sparse random graphs.
\newblock {\em Discrete Math.}, 59(3):243--256, 1986.

\bibitem{F14}
A.~Frieze.
\newblock Random structures and algorithms.
\newblock In {\em Proceedings of the {I}nternational {C}ongress of {M}athematicians (ICM 2014), Seoul, Korea, August 13--21, 2014. Vol. I: Plenary lectures and ceremonies}, pages 311--340. Seoul: KM Kyung Moon Sa, 2014.

\bibitem{F19}
A.~Frieze.
\newblock Hamilton {Cycles} in {Random} {Graphs}: a bibliography.
\newblock {\em arXiv:1901.07139}, 2019.

\bibitem{J04}
J.~R. Johnson.
\newblock Long cycles in the middle two layers of the discrete cube.
\newblock {\em J. Comb. Theory, Ser. A}, 105(2):255--271, 2004.

\bibitem{K23}
M.~Krivelevich.
\newblock Component sizes in the supercritical percolation on the binary cube.
\newblock {\em arXiv preprint arXiv:2311.07210}, 2023.

\bibitem{L22}
L.~Lichev.
\newblock The giant component after percolation of product graphs.
\newblock {\em J. Graph Theory}, 99(4):651--670, 2022.

\bibitem{L93}
L.~Lov\'{a}sz.
\newblock {\em Combinatorial problems and exercises}.
\newblock North-Holland Publishing Co., Amsterdam, second edition, 1993.

\bibitem{M16}
T.~M{\"u}tze.
\newblock Proof of the middle levels conjecture.
\newblock {\em Proc. Lond. Math. Soc. (3)}, 112(4):677--713, 2016.

\bibitem{M24}
T.~M{\"u}tze.
\newblock A book proof of the middle levels theorem.
\newblock {\em Combinatorica}, 44(1):205--208, 2024.

\bibitem{S67}
A.~A. Sapo\v{z}enko.
\newblock Metric properties of almost all functions of the algebra of logic.
\newblock {\em Diskret. Analiz}, 10:91--119, 1967.

\bibitem{W96}
D.~B. West.
\newblock {\em Introduction to graph theory}.
\newblock Prentice Hall, Inc., Upper Saddle River, NJ, 1996.

\end{thebibliography}

\end{document}